%% file: On_the_existence_of_Euler-Lagrange_orbits....tex
\documentclass[reqno]{amsart}
\usepackage{amsmath}
\usepackage{amsfonts}
\usepackage{amsthm}
\usepackage{amssymb}
\usepackage{verbatim}
\usepackage{enumerate}
\usepackage{graphicx}
\usepackage{longtable}
\usepackage[all,cmtip]{xy}
\usepackage{upref}
\usepackage{fixltx2e}
\usepackage{hyperref}
\usepackage[hang,small,bf]{caption}
\usepackage{tikz}

\newtheorem{defn}{Definition}[section]

\newtheorem{teo*}{Theorem}
\newtheorem{cor*}{Corollary}

\newtheorem{prop}[defn]{Proposition}
\newtheorem{oss}[defn]{Remark}
\newtheorem{lemma}[defn]{Lemma}
\newtheorem{cor}[defn]{Corollary}

\newtheorem*{question}{Question}
\newcommand{\bigslant}[2]{{\raisebox{.2em}{$#1$}\left/\raisebox{-.2em}{$#2$}\right.}}

\numberwithin{equation}{section}

\renewcommand{\theta}{\vartheta}

\newcommand{\T}{{\mathbb T}}
\newcommand{\Z}{{\mathbb Z}}
\newcommand{\R}{{\mathbb R}}
\newcommand{\N}{{\mathbb N}}

\newcommand{\A}{{\mathbb A}}
\newcommand{\HH}{{\mathbb H}}

\renewcommand{\phi}{\varphi}

\begin{document}

\title{On the existence of Euler-Lagrange orbits satisfying the conormal boundary conditions.}
\author{Luca Asselle}
\address{Ruhr Universit\"at Bochum, Fakult\"at f\"ur Mathematik\newline\indent Geb\"aude NA 4/33, D-44801 Bochum, Germany}
\email{luca.asselle@ruhr-uni-bochum.de}
\date{September 7, 2016}

\begin{abstract}
Let $(M,g)$ be a closed connected Riemannian manifold, $L:TM\rightarrow \R$ be a Tonelli Lagrangian. Given two closed submanifolds $Q_0,Q_1\subseteq M$ and a real number $k$, we study the existence of Euler-Lagrange orbits with energy $k$ 
connecting $Q_0$ to $Q_1$ and satisfying suitable boundary conditions, known as \textit{conormal boundary conditions}. We introduce the Ma\~n\'e critical value which is relevant for this problem and discuss existence results for supercritical and subcritical energies. 
We also provide counterexamples showing that all the results are sharp.
\end{abstract}

\maketitle


\section{Introduction}
\label{introduction}

Let $(M,g)$ be a closed connected Riemannian manifold and let $L:TM\rightarrow \R$ be a Tonelli Lagrangian (that is a smooth fiberwise $C^2$-strictly convex and superlinear function). The Euler-Lagrange equation, which in local coordinates is given by 
$$\frac{d}{dt} \frac{\partial L}{\partial v}(\gamma,\dot \gamma) - \frac{\partial L}{\partial q} (\gamma,\dot \gamma) = 0 ,$$
gives rise to a flow on $TM$, known as the \textit{Euler-Lagrange flow}. The energy function
$$E(q,v) =  d_vL(q,v)\cdot v - L(q,v)$$ 
associated with $L$ is a prime integral of the motion, meaning that it is constant along solutions of the Euler-Lagrange equation. Moreover, $E$ is Tonelli and attains its minimum at $v=0$;
in particular, the energy level sets $E^{-1}(k)$ are compact and invariant under the Euler-Lagrange flow, which therefore turns out to be complete on $TM$. Here we are interested in the following

\begin{question}
Given two non-empty closed submanifolds $Q_0,Q_1\subseteq M$, for which $k\in \R$ does there exist an Euler-Lagrange orbit $\gamma$ with energy $k$ and satisfying  the \textit{conormal boundary conditions}?
\end{question}

Without loss of generality we may suppose $Q_0,Q_1$ connected. Recall that an Euler-Lagrange orbit $\gamma:[0,R]\rightarrow M$ is said to satisfy the \textit{conormal boundary conditions} if
\begin{equation}
\left\{ \begin{array}{l} \gamma(0)\in Q_0, \ \gamma(R) \in Q_1, \\ d_vL(\gamma(0),\dot \gamma(0)) \Big |_{T_{\gamma(0)}Q_0} = 0 , \\  d_vL(\gamma(R),\dot \gamma(R)) \Big |_{T_{\gamma(R)}Q_1}= 0.\end{array}\right .
\label{cbc}
\end{equation}

In the case of geodesic flows (i.e.~when $L$ is just the kinetic energy defined by $g$), one is simply requiring that $\gamma$ is a geodesic hitting $Q_0$ and $Q_1$ orthogonally.
For sake of conciseness, throughout the paper we will call solutions of \eqref{cbc} simply \textit{connecting orbits}. 

The question above can also be formulated in the Hamiltonian setting. Let  $H:T^*M\rightarrow \R$ be the Tonelli Hamiltonian given by the Fenchel dual of $L$ 
\begin{equation}
H(q,p) = \max_{v\in T_qM} \Big [\langle p,v\rangle_q - L(q,v)\Big ],
\label{magnetichamiltonian}
\end{equation}
where $\langle \cdot,\cdot \rangle$ denotes the duality pairing between tangent and cotangent bundle. For which $k\in \R$ does $H^{-1}(k)$ carry a Hamiltonian orbit $u:[0,R]\rightarrow T^*M$ with
$$u(0) \in N^*Q_0 , \quad u(R) \in N^*Q_1 ?$$
Here, for $i=0,1$, $N^*Q_i$ is the \textit{conormal bundle} of $Q_i$  
$$N^*Q_i := \Big \{(q,p)\in T^* M \ \Big | \ q\in Q_i, \ T_q Q_i\subseteq \ker p \Big \} .$$
We refer to  \cite{AS09}, \cite{Dui76} or \cite[Section 6.4]{Hör90} for general facts and properties of conormal bundles.

\begin{oss} It follows from the Hamiltonian formulation that a necessary condition for the existence of connecting orbits is that the energy level set $H^{-1}(k)$ intersects both the conormal bundles of $Q_0$ and $Q_1$, namely 
\begin{equation}
H^{-1}(k) \cap N^*Q_i \neq \emptyset , \quad \text{for} \ i=0,1 .
\label{necessaryham}
\end{equation}
We set 
$$k(L;Q_0,Q_1):=\inf \Big \{k\in \R \ \Big |\ \eqref{necessaryham} \text{ holds}\Big \}.$$

The above observation can be phrased by saying that $k\geq k(L;Q_0,Q_1)$ is a necessary condition for the existence of connecting orbits.
This condition alone is however not sufficient, as we will show in Section \ref{counterexamples}.
\end{oss}

\vspace{2mm}

A particular class of Tonelli Lagrangians is given by the so-called \textit{magnetic Lagrangians}, i.e. smooth functions on $TM$ of the form 
\begin{equation}
L(q,v) =  \frac12 |v|^2 + \theta_q(v),
\label{magneticlag}
\end{equation}
where $|\cdot |$ is the norm induced by the Riemannian metric $g$ and $\theta$ is a smooth one-form on $M$. The reason for this terminology is that they can be thought of as 
modelling the motion of a unitary mass and charge particle under the effect of the magnetic field $\sigma = d\theta$.

In the Lagrangian setting, condition \eqref{necessaryham} for a magnetic Lagrangian is expressed by 
\begin{equation}
k \geq \max \left \{ \min_{q\in Q_0} \frac12  |\mathcal P_0w_q|^2, \min_{q\in Q_1} \frac12 |\mathcal P_1w_q|^2 \right \} , 
\label{necessarylag}
\end{equation}
where $\mathcal P_i:TM|_{Q_i}\rightarrow TQ_i$ denotes the orthogonal projection and $w_q\in T_qM$ is the unique tangent vector representing $\theta_q\in T^*_qM$. The right-hand side of \eqref{necessarylag} is precisely
$k(L;Q_0,Q_1)$; if it is non-zero, then we cannot expect existence of connecting orbits for every positive energy, even if the submanifolds intersect or if $Q_0=Q_1$.

\begin{oss}
When $Q_0=\{q_0\}$ and $Q_1=\{q_1\}$ are points in $M$, the question above reduces to the problem of finding those energy levels which contain Euler-Lagrange orbits connecting $q_0$ and $q_1$. This problem has 
an easy answer when $L$ is a mechanical Lagrangian, i.e. of the form
\begin{equation}
L(q,v) = \frac12 |v|^2 - V(q),
\label{mechaniclag}
\end{equation}
with $V$ smooth function on $M$ (potential energy), but is made extremely hard by the presence of a magnetic potential $\theta$ (see e.g. \cite[Chapter I.3 and Appendix F]{Gli97}). We will get back on this later on in this introduction and 
in the last section.
\end{oss}

\begin{oss}
The conormal boundary conditions \eqref{cbc} make also sense for submanifolds of $M\times M$ which are not necessarily of the form $Q_0\times Q_1$. In this sense, the problem of finding periodic orbits of the Euler-Lagrange 
flow can be viewed as the problem of finding Euler-Lagrange orbits satisfying the conormal boundary conditions for $\Delta \subseteq M\times M$ diagonal.
\end{oss}

The key fact that will be exploited throughout this paper is that connecting orbits with energy $k$ correspond to the critical points of the \textit{free-time Lagrangian action functional}
$$\A_k:\mathcal M_Q \longrightarrow \R, \quad \A_k(x,T)= T \int_0^1 \Big [L\Big (x(s),\frac{ x'(s)}{T} \Big ) + k \Big ]\, ds ,$$
where $Q=Q_0\times Q_1$ and $\mathcal M_Q =H^1_Q([0,1],M)\times (0,+\infty)$ is the Hilbert manifold of $H^1$-paths connecting $Q_0$ with $Q_1$ with arbitrary interval of definition. When $Q_0\cap Q_1\neq \emptyset$, we identify $Q_0\cap Q_1$ with the subset of $\mathcal M_Q$ made of constant loops with values in $Q_0\cap Q_1$.

Notice that $\A_k$ is well-defined only 
under the additional assumption that $L$ is quadratic at infinity; this is however not a problem for our purpose. Indeed, since the energy level $E^{-1}(k)$ is compact, we can always modify $L$ outside it to achieve the quadratic growth condition.
Hereafter all the Lagrangians will be thus supposed without loss of generality to be quadratic at infinity.

The goal of the present work will be to see under which assumptions the existence of critical points for $\A_k$ is guaranteed. It is clear that in this study a crucial role will be played by the analytical (e.g. ``compactness'' and the Palais-Smale condition) and geometric properties (e.g. boundedness or the 
presence of a mountain-pass geometry) of $\A_k$, as well as by the topological properties of the space $\mathcal M_Q$. However, if on the one hand the topology of $\mathcal M_Q$ clearly do not depend on $k$, on the other hand 
the properties of $\A_k$ change drastically when crossing a suitable energy value. This is actually no surprise, since also the dynamical and geometric properties of the Euler-Lagrange flow change when crossing suitable
\textit{Ma\~n\'e critical values} (cf. \cite{Abb13,Con06}). 

In general, the critical points for $\A_k$ one might expect to find are either 
\begin{itemize}
\item global (or local) minimizers,
\end{itemize}
or 
\begin{itemize}
\item mountain passes (or more generally minimax critical points).
\end{itemize}

\noindent The questions one has to address are therefore the following 

\begin{enumerate}
\item For which $k$ is $\A_k$ bounded from below on the connected components of $\mathcal M_Q$? And for those values of $k$, on which connected components of $\mathcal M_Q$ does then $\A_k$ admit minimizers? 
\vspace{2mm}
\item Assume that the topology of $\mathcal M_Q$ (or the geometry of $\A_k$) allows to define a suitable minimax class. For which $k$ does this then yield existence of critical points for $\A_k$?
\end{enumerate}

We shall however observe already at this point that a general existence result of critical points for $\A_k$ cannot be obtained since there are examples of Euler-Lagrange flows and of submanifolds $Q_0,Q_1$ for which there are no connecting orbits. Consider for instance the geodesic
flow on the flat Torus $(\T^2,g_{flat})$ and pick $Q_0,Q_1$ as in Figure \ref{figure1} below.

\begin{figure}[h]
\begin{center}
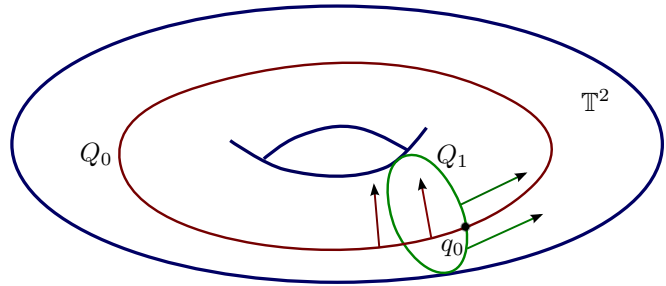
\end{center}
\caption{Existence of connecting orbits might fail in full generality.}
\label{figure1}
\end{figure}

The conormal boundary conditions \eqref{cbc} imply that geodesics connecting $Q_0$ with $Q_1$ have to hit both $Q_0$ and $Q_1$ orthogonally, which is not possible (except for the constant solution in the intersection point $q_0$). It follows that for every $k>0$ there 
are no geodesics connecting $Q_0$ with $Q_1$ and satisfying the conormal boundary conditions. From a variational viewpoint, what goes wrong in this example is that $\mathcal M_Q$ is connected, contains constant paths and $\pi_l(\mathcal M_Q,\{q\})$ is trivial for every $l\in \N$, for $\mathcal M_Q$ is contractible. This implies that 
$\A_k$ has infimum zero on $\mathcal M_Q$, for every $k>0$, and this is not attained. Also, one cannot expect to find minimax critical points, since there are no non-trivial minimax classes to play with.

This is however (under some mild assumption on the intersection if $Q_0\cap Q_1\neq \emptyset$) the only possible counterexample, at least if $k$ is ``sufficiently large'' as we now briefly explain. Let $G<\pi_1(M)$ be the smallest normal subgroup containing both $\imath_*(\pi_1(Q_0))$ 
and $\imath_*(\pi_1(Q_1))$, where $\imath:Q_i\hookrightarrow M$ is the canonical inclusion; consider the cover $M'=\widetilde M/G$, with $\widetilde M$ universal cover of $M$, and define
\begin{equation}
c(L;Q_0,Q_1):= \inf_{u\in C^\infty(M')} \sup_{q\in M'} \ H'(q,d_qu),
\label{clq0q1intro}
\end{equation}
where $H'$ is the lift to $M'$ of the Hamiltonian $H:T^*M\rightarrow \R$ associated with $L$. 

\begin{oss}
We have $c(L;Q_0,Q_1)\geq k(L;Q_0,Q_1)$ (this will be proved in Proposition \ref{klminorecl}). Moreover, if $Q_0=\{q_0\}$, $Q_1=\{q_1\}$, then 
\begin{align*}
k(L;\{q_0\},\{q_1\}) =\max \{E(q_1,0),E(q_2,0)\},\quad c(L;\{q_0\},\{q_1\}) =c_u(L).
\end{align*} 
Here $c_u(L)$ is the Ma\~n\'e critical value of the universal cover and is defined as in \eqref{clq0q1intro} replacing $M'$ with the universal cover 
of $M$.
\end{oss}

In Section \ref{themanecriticalvalueclq0q1} we will show that, for $k\geq c(L;Q_0,Q_1)$, $\A_k$ is bounded from below on each connected component of $\mathcal M_Q$ and it is unbounded from below on each connected component otherwise.
Furthermore, for $k> c(L;Q_0,Q_1)$ every Palais-Smale sequence for $\A_k$ with times bounded away from zero admits converging subsequences. These facts will allow us in Section 
\ref{existenceresultsforhighenergies} to prove the following:

\begin{teo*}
Let $\mathcal N$ be a connected component of $\mathcal M_Q$; then we have:
\begin{enumerate}
\item If $\mathcal N$ does not contain constant paths, then for all $k>c(L;Q_0,Q_1)$ there exists a global minimizer of $\A_k|_{\mathcal N}$.
\item Suppose now that $\mathcal N$ contains constant paths and define 
$$k_{\mathcal N}(L):= \sup \Big \{ k\in \R \ \Big |\ \inf_{\mathcal N} \A_k < 0 \Big \} \in [c(L;Q_0,Q_1),+\infty).$$
Then the following hold:
\begin{enumerate}
\item  For all $k\in (c(L;Q_0,Q_1), k_{\mathcal N}(L))$ there is a global minimizer of $\A_k|_{\mathcal N}$. 
\item If $Q_0\cap Q_1$ is connected, $\mathcal N$ has the retraction property (see the beginning of Section \ref{existenceresultsforhighenergies} for 
the definition), and $\pi_l(\mathcal N,Q_0\cap Q_1)\neq 0$ for some $l\geq 1$, then for all $k>k_{\mathcal N}(L)$ there is a minimax critical point for $\A_k|_{\mathcal N}$.
\item If $Q_0\cap Q_1$ is not connected and at least one of its connected component is isolated, then for all $k>k_{\mathcal N}(L)$ there exists a minimax critical point for $\A_k|_{\mathcal N}$.
\end{enumerate}
\end{enumerate}
\label{teosupercriticalintro}
\end{teo*}

In the following we refer to \textit{supercritical energies} whenever $k>c(L;Q_0,Q_1)$ and to \textit{subcritical energies} whenever $k\in (k(L;Q_0,Q_1),c(L;Q_0,Q_1))$. 

A particular case of intersecting submanifolds is given by the choice $Q_0=Q_1$, which corresponds to (a particular case of) the \textit{Arnold chord conjecture} 
about the existence of a Reeb orbit starting and ending at a given Legendrian submanifold of a contact manifold, see \cite{Arn86, Moh01}, but in a possibly \textit{virtually} contact situation (c.f. introduction of \cite{CFP10} for the definition), since in general $c(L;Q_0,Q_0)$ might be strictly lower than $c_0(L)$. Recall indeed that energy levels 
above $c_u(L)$ are virtually contact (cf. \cite[Lemma 5.1]{CFP10}), however they are known to be not of contact type if $c_u(L)<k<c_0(L)$ (cf. \cite[Proposition B.1]{Con06}). 
Here $c_0(L)$ is defined as in \eqref{clq0q1intro} replacing $M'$ with the abelian cover of $M$ and is called the \textit{Ma\~n\'e critical value 
of the abelian cover}. 

In our setting an \textit{Arnold chord} is simply an Euler-Lagrange orbits starting and ending at $Q_0$ and satisfying the conormal boundary conditions.

\begin{cor*}
Let $Q_0\subseteq M$ be a non-empty closed connected submanifold and define $c(L;Q_0)$ as in (\ref{clq0q1intro}) just by setting $Q_0=Q_1$. Then the following hold:
\begin{enumerate}
\item For every $k>c(L;Q_0)$ and for every connected component of $\mathcal M_Q$ that does not contain constant paths there exists an Arnold chord 
with energy $k$ which is a global minimizer of $\A_k$ among its connected component.
\item Let $\mathcal N$ be the connected component of $\mathcal M_Q$ containing the constant paths. 
For every $k\in (c(L;Q_0), k_{\mathcal N} (L))$ there is an Arnold chord with energy $k$ which is a global minimizer of $\A_k$ on $\mathcal N$. 
Moreover, if $\pi_l(M,Q_0)\neq \{0\}$ for some $l\geq 2$, then for every $k> k_{\mathcal N}(L)$ there exists an Arnold chord in $\mathcal N$ with energy $k$.
\end{enumerate}

In particular, if $Q_0\neq M$, then for all $k>c(L;Q_0)$, $k\neq k_{\mathcal N}(L)$, there is an Arnold chord with energy $k$.
\label{arnoldchordsupercriticalintro}
\end{cor*}

Existence results for subcritical energies are harder to achieve than the corresponding ones for supercritical energies and the reason for that are of various nature. 

First, when $k<c(L;Q_0,Q_1)$, $\A_k$ could have Palais-Smale sequences with times going to infinity. In fact, the lack of the Palais-Smale condition for subcritical energies is ultimately 
responsible for the fact that one gets existence results which hold only for \textit{almost every energy} in a suitable range of subcritical energies. 

Second, in case the intersection $Q_0\cap Q_1$ is empty the problem might have no solutions for every $k\in (k(L;Q_0,Q_1),c(L;Q_0,Q_1))$ as it contains, as a very special case, the problem of finding the energy levels for which any two points in $M$ can be joined by an Euler-Lagrange 
orbit. In this direction it has been proven by Ma\~ne in \cite[Page 151]{Man96} that, for every $k>c_0(L)$, every pair of points in $M$ can be joined by an Euler-Lagrange orbit. This result has been strenghtened by Contreras in \cite{Con06} to every $k>c_u(L)$. 
In Section \ref{counterexamples} we provide an example showing that Contreras' result, and Theorem \ref{teosupercriticalintro} as well, are actually sharp.

\begin{teo*}
There exist a Tonelli Lagrangian $L:T\T^2\rightarrow \R$ and two disjoint submanifolds $Q_0,Q_1\subseteq \T^2$ such that $c_u(L)<c(L;Q_0,Q_1)$ and with
no connecting orbits having energy $k\leq c(L;Q_0,Q_1)$. Moreover, there are points $q_0\in Q_0$ and $q_1 \in Q_1$ such 
that there are no Euler-Lagrange orbits connecting them with energy $k\leq c_u(L)$. 
\label{counterexample1intro}
\end{teo*}

In order to get existence results for subcritical energies one has therefore to assume that $Q_0\cap Q_1 \neq \emptyset$. Denote with $\mathcal N$ the connected component of $\mathcal M_Q$ containing the constant paths and with $\Omega$ an isolated connected component 
of $Q_0\cap Q_1$, meaning that there exists $\epsilon >0$ such that $B_\epsilon (\Omega)$ is disjoint from any other connected components of $Q_0\cap Q_1$. Now set
$$k_{\Omega} := \min \left \{c(L;Q_0,Q_1), \max_{q\in \Omega} E(q,0) + \lambda \cdot \max_{q\in \Omega} |d_vL(q,0)|^2\right \} ,$$
where $\lambda>0$ is a constant depending only on $L$ which equals $\frac 12$ in case $L$ is a magnetic Lagrangian (see Section \ref{existenceresultsforlowenergies}
for the precise definition). The definition of $k_\Omega$ in case $\Omega$ is not isolated is more delicate and will be postponed to Section \ref{existenceresultsforlowenergies}.

\begin{oss}
The energy values $k(L;Q_0,Q_1)$ and $k_\Omega$ strongly depend on how the submanifolds $Q_0$ and $Q_1$ sit inside $M$. This is in sharp contrast with what 
happens for the critical value $c(L;Q_0,Q_1)$ which only depends on the homotopy classes of $(M,Q_0)$ and $(M,Q_1)$, meaning that if we have a continuous map $\{F_t:(M,Q_0^0)\rightarrow (M,Q_0^t)\}_{t\in I}$, then
$$c(L;Q_0^0,Q_1)=c(L;Q_0^t,Q_1),\quad \forall t\in I.$$
Moreover it follows directly from the definition that $c(L;Q_0,Q_1)=c(L;Q_1,Q_0)$.
\end{oss}

In general the energy value $k_{\Omega}$ need not coincide with $c(L;Q_0,Q_1)$; examples will be provided in Section \ref{counterexamples}. 
The relevance of $k_{\Omega}$ relies on the fact that, for $k\in (k_{\Omega},c(L;Q_0,Q_1))$ the free-time action functional has a mountain-pass geometry on $\mathcal N$. The two valleys are represented by the set of constant paths 
and by the the set of paths with negative action (which is non-empty as $k<c(L;Q_0,Q_1))$. 
Exploiting this mountain-pass geometry we get the following

\begin{teo*}
For almost every $k\in (\displaystyle \inf_{\Omega} k_{\Omega},c(L;Q_0,Q_1))$ there is a connecting orbit with energy $k$.
\label{teosubcriticalintro}
\end{teo*}

In the theorem above, by taking the infimum of $k_\Omega$ over all connected components $\Omega$ of $Q_0\cap Q_1$ we get a critical value which is a priori smaller than $k_{Q_0\cap Q_1}$
and, hence, a sharper result. 

The ``almost every'' relies exactly on the lack of the Palais-Smale condition for $\A_k$ for subcritical energies. To overcome this difficulty one has to use an argument originally due to Struwe \cite{Str90}, which has already been intensively applied 
to the existence of periodic Euler-Lagrange orbits for subcritical energies \cite{Abb13,AMMP14,AB14,AB15a,Con06}, called the \textit{Struwe monotonicity argument}.

As a trivial corollary we get the following existence result of Arnold chords for subcritical energies. 
\begin{cor*}
Let $Q_0\subseteq M$ be a non-empty closed connected submanifold. Then for almost every $k\in (k_{Q_0}, c(L;Q_0))$ there is an Arnold chord 
with energy $k$. 
\label{arnoldchord}
\end{cor*}

As we will see in Section \ref{counterexamples}, also Theorem \ref{teosubcriticalintro} above is sharp. 

\begin{teo*}
For every closed surface $\Sigma$, there exist a Tonelli Lagrangian $L:T\Sigma \rightarrow \R$ and intersecting submanifolds $Q_0,Q_1\subseteq \Sigma$ such 
that 
$$k(L;Q_0,Q_1)<k_\Omega<c(L;Q_0,Q_1)$$
and with no connecting orbits having energy $k<k_\Omega$.
\label{counterexample2intro}
\end{teo*}

We end this introduction giving a brief summary of the contents of the paper:

\begin{itemize}
\item In Section \ref{thefreetimelagrangianactionfunctional} we introduce the free-time Lagrangian action functional $\A_k$ rigorously and discuss its properties (with particular attention to the Palais-Smale condition and to the completeness of the negative gradient flow). 
\item In Section \ref{themanecriticalvalueclq0q1} we define the Ma\~n\'e critical value $c(L;Q_0,Q_1)$ which is relevant for the problem and show how the properties of $\A_k$ change when considering ``subcritical'' rather than ``supercritical'' values of $k$. 
\item In Section \ref{existenceresultsforhighenergies} we deal with the case of supercritical energies and prove Theorem \ref{teosupercriticalintro}.
\item In Section \ref{existenceresultsforlowenergies} we consider the case of subcritical energies and prove Theorem \ref{teosubcriticalintro}.
\item Finally, in Section \ref{counterexamples}, we prove Theorems \ref {counterexample1intro} and \ref{counterexample2intro}.
\end{itemize}


\section{The free-time Lagrangian action functional}
\label{thefreetimelagrangianactionfunctional}

For any given absolutely continuous curve $\gamma:[0,T]\rightarrow M$ we define $x:[0,1]\rightarrow M$ as $x(s):=\gamma(s\, T)$. Throughout the whole work we will identify $\gamma$ with the pair $(x,T)$.

To avoid confusion we will always denote with a \textit{dot} the derivative with respect to $t$ and with a \textit{prime} the derivative with respect to $s$.

Fix a real number $k$, the value of the energy for which we would like to find Euler-Lagrange orbits satisfying the conormal boundary conditions \eqref{cbc}.
Recall that, since the energy level $E^{-1}(k)$ is compact, up to the modification of $L$ outside it, we may assume the Tonelli Lagrangian $L$ to be quadratic at infinity. In particular
\begin{align}
L(q,v) &\geq a |v|^2 - b , \quad \forall (q,v)\in TM , \label{firstinequality} \\ 
d_{vv}L(q,v)[u,u] &\geq 2a |u|^2, \quad  \forall (q,v)\in TM ,\ \forall u\in T_qM, \label{secondinequality}
\end{align}
for suitable numbers $a >0$, $b\in \R$ and
\begin{equation}
S_k(x,T) := \int_0^T \Big[L(\gamma(t),\dot{\gamma}(t))+k \Big ]\, dt= T\int_0^1 \Big[ L\Big(x(s),\frac{x'(s)}{T} \Big ) + k\Big ]\, ds 
\label{freetimeactionfunctional}
\end{equation}
is well-defined for every $x\in H^1([0,1],M)$. Hence, we get a well-defined functional
$$S_k : H^1([0,1],M) \times (0,+\infty) \longrightarrow \R\, ,$$
called the \textit{free-time action functional}. The domain of definition $\mathcal M:=H^1([0,1],M)\times (0,+\infty)$ of $S_k$ can be interpreted as the space of $H^1$-paths in $M$ 
with arbitrary interval of definition through the identification $\gamma=(x,T)$ above and it has a natural structure of product Hilbert manifold given by the product metric 
\begin{equation}
g_{\mathcal M} := g_{H^1}+ dT^2,
\label{productmetric}
\end{equation}
where $g_{H^1}$ is the standard metric on $H^1([0,1],M)$ induced by the given Riemannian metric $g$ on $M$ (see \cite{AS09} for further details). Obviously, $(\mathcal M,g_{\mathcal M})$ is not complete as the factor $(0,+\infty)$ is 
not complete with respect to the Euclidean metric. The following proposition is about the regularity of the free-time action functional $S_k$; for the proof we refer again to \cite{AS09} (see also \cite[Proposition 3.1.1]{Ass15}).

\vspace{1mm}

\begin{prop}
The following hold:
\begin{enumerate}
\item $S_k \in C^{1,1}(\mathcal M)$ and it has second Gateaux differential at every point.
\item $S_k$ is twice Fr\'ech\'et differentiable at every point if and only if $L$ is electromagnetic on the whole $TM$; in this case, $S_k$ is actually smooth.
\end{enumerate}
\label{propregularity}
\end{prop}

Let now $Q_0,Q_1\subseteq M$ be non-empty closed connected submanifolds. Since we want to prove the existence of connecting orbits, we shall consider the restriction of $S_k$ to the smooth submanifold 
$$\mathcal M_Q:= H^1_Q([0,1],M) \times (0,+\infty),$$
where $Q=Q_0\times Q_1$ and $H^1_Q([0,1],M)$ is the space of $H^1$-paths $x:[0,1]\rightarrow M$ connecting $Q_0$ with $Q_1$. We denote with $\A_k$ the restriction $S_k|_{\mathcal M_Q}$.
The importance of $\A_k$ relies on the following:

\begin{prop}
A curve $\gamma=(x,T)$ is a connecting orbit with energy $E(\gamma,\dot{\gamma})= k$ 
if and only if $(x,T)$ is a critical point of the free-time action functional $\A_k$.
\label{equivalence}
\end{prop}

\begin{proof}
The pair $(x,T)$ is a critical point for $\A_k$ if and only if 
$$d\A_k(x,T)[(\zeta,H)]= 0$$ 
for any choice of $(\zeta,H)$. It is well-known (see e.g. \cite{AS09}) that the condition 
$$d_x\A_k(x,T)[(\zeta,0)]=0$$
is equivalent to $\gamma(t):=x(t/T)\in H^1_Q([0,T],M)$ being an Euler-Lagrange orbit satisfying the conormal boundary conditions (\ref{cbc}). Furthermore, a simple computation shows that
\begin{equation}
\frac{\partial \A_k}{\partial T}(x,T)\ =\ \frac{1}{T}\int_0^T \Big [k-E(\gamma(t),\dot{\gamma}(t))\Big ]\, dt
\label{derivatainT}
\end{equation}
which implies that $E(\gamma,\dot{\gamma})= k$, since the energy is constant along $\gamma$.
\end{proof}

\vspace{4mm} 

\noindent \textbf{Completeness properties for $\A_k$.} Since the Hilbert manifold $\mathcal M_Q$ is not complete, it is useful to know whether sublevel sets of the free-time action functional $\A_k$ are complete or not. With the next lemma we see that completeness on a given connected component $\mathcal N$ of $\mathcal M_Q$ only depends 
on the fact that $\mathcal N$ contains constant paths or not. 

\begin{lemma}
The following statements hold:

\begin{enumerate}
\item The sublevel sets of $\A_k$  in each connected component $\mathcal N$ of $\mathcal M_Q$  not containing constant paths are complete.
\item If $(x_h,T_h)$ is such that $T_h\rightarrow 0$, then
\begin{equation}
\liminf_{h\rightarrow +\infty} \A_k(x_h,T_h) \geq 0.
\label{liminf}
\end{equation}
\end{enumerate}
\label{completeness0}
\end{lemma}

\begin{proof}
By (\ref{firstinequality}) we have the chain of inequalities
\begin{align}
\A_k(x,T) &= T \int_0^1 \Big  [L\Big (x(s),\frac{x'(s)}{T}\Big ) + k\Big ] ds \nonumber  \\
		&\geq  T\int_0^1 \Big [ a \, \frac{|x'(s)|^2}{T^2} - b + k \Big ] ds  \nonumber \\
                &= \frac{a}{T} \int_0^1 |x'(s)|^2 \, ds +T(k-b) \nonumber \\ 
                &\geq \frac{a}{T} l(x)^2 + T(k-b) \label{stimaazione}
\end{align}
where $l(x)$ denotes the length of the path $x$. Since $\mathcal N$ does not contain constant paths, the length of any path in $\mathcal N$ is bounded away from zero 
by a suitable positive constant. Therefore, $T$ is bounded away from zero on
$$\Big \{(x,T) \in \mathcal N\  \Big | \ \A_k(x,T) \leq c\Big \}$$
for any $c\in \R$, proving the statement.

Inequality (\ref{stimaazione}) actually also proves the second statement. In fact, if $ T_h\rightarrow 0$ then 
$$ T_h(k-b) \longrightarrow 0 , \quad \text{for} \  h\longrightarrow +\infty$$
and hence the action $\A_k(x_h,T_h)$ is eventually bigger than $-\epsilon$, for arbitrary $\epsilon>0$.
\end{proof}

\vspace{2mm}

\begin{cor}
If $c<0$, then the sublevel set $\{\A_k\leq c\}$ is complete.
\label{cornegative}
\end{cor}

\begin{proof}
Follows directly from Statement 2 in Lemma \ref{completeness0}.
\end{proof}

\vspace{2mm}

We end this section studying the possible sources of non-completeness of the negative gradient flow of $\A_k$ on $\mathcal M_Q$.  Up to changing $-\nabla \A_k$ with 
the conformally equivalent bounded vector field
$$- \frac{\nabla \A_k}{\sqrt{1+|\nabla \A_k|^2}}\, $$
we may assume the negative gradient flow to be complete on every connected component of $\mathcal M_Q$ not containing constant paths. Also, on the connected components containing constant paths, 
incompleteness occurs only if there are flow-lines for which $T(\cdot) \rightarrow 0$ in finite time. The next lemma ensures that, for such flow lines, $\A_k$ necessarily goes to zero.

\begin{lemma}
Let $\big (x(\cdot),T(\cdot)\big ):[0,\sigma^*)\rightarrow \mathcal M_Q$ be a negative gradient flow-line with 
$$\liminf_{\sigma\rightarrow \sigma^*} T(\sigma)= 0 .$$
\noindent Then 
$$\lim_{\sigma \rightarrow \sigma^*} \A_k\big (x(\sigma),T(\sigma)\big )= 0 .$$
\label{convergenza}
\end{lemma}

\vspace{-5mm}

\begin{proof}
The proof is analogous to the one of Lemma 3.3 in \cite{Abb13}, where the case of periodic orbits is considered.
Since both $E$ and $L$ are quadratic at infinity, we have
$$E(q,v)\geq c_0  L(q,v)-c_1$$
for some $c_0>0$ and $c_1\in \R$. Therefore from (\ref{derivatainT}) it follows that
\begin{align*}
\frac{\partial \A_k}{\partial T} (x,T) &= \frac{1}{T}\int_0^T \Big [k-E(\gamma(t),\dot{\gamma}(t))\Big ]\, dt \\
&\leq \frac{1}{T} \int_0^T \Big [k-c_0\, L(\gamma(t),\dot{\gamma}(t))+c_1\Big ]\, dt\\
&= (c_0+1) k + c_1 - \frac{c_0}{T} \A_k(x,T)
\end{align*}
and hence 
\begin{equation}
\A_k(x,T) \leq \frac{T}{c_0}\, \Big [(c_0+1) k + c_1 - \frac{\partial \A_k}{\partial T}(x,T)\Big ]= \frac{T}{c_0}  \Big [C-\frac{\partial \A_k}{\partial T}(x,T)\Big ] ,
\label{stima}
\end{equation}
where $C$ is a suitable constant. By assumption, there is a sequence $\sigma_h\uparrow \sigma^*$ with 
$$T'(\sigma_h) \leq 0 , \quad  T(\sigma_h)\longrightarrow 0 .$$
Since $\sigma\mapsto (x(\sigma),T(\sigma))$ is a negative gradient flow-line, we have
$$0\geq T'(\sigma_h)=-\frac{\partial \A_k}{\partial T}\big (x(\sigma_h),T(\sigma_h)\big )$$ 
and hence 
$$\A_k\big (x(\sigma_h),T(\sigma_h)\big ) \leq \frac{T(\sigma_h)}{c_0} \Big [C-\frac{\partial \A_k}{\partial T}\big ( x(\sigma_h),T(\sigma_h)\big )\Big ] \leq \frac{C}{c_0}\, T(\sigma_h) .$$
Since $T(\sigma_h)\rightarrow 0$, from the inequality above we deduce that
$$\limsup_{h\rightarrow +\infty} \A_k\big (x(\sigma_h),T(\sigma_h)\big ) \leq 0 .$$ 

The assertion follows now from Statement 2 in Lemma \ref{completeness0} and from the monotonicity of the function $\sigma \mapsto \A_k(x(\sigma),T(\sigma))$.
\end{proof}

\vspace{2mm}

\noindent \textbf{The Palais-Smale condition for $\A_k$.} Recall that a Palais-Smale sequence at level $c$ for $\A_k$ is a sequence $(x_h,T_h)\subseteq \mathcal M_Q$ such that 
$$\A_k(x_h,T_h)\longrightarrow c , \quad|d\A_k(x_h,T_h)|\longrightarrow 0 ,$$
where $|\cdot|$ denotes the norm on $T^*\mathcal M_Q$ induced by the Riemannian metric $g_{\mathcal M}$ in \eqref{productmetric}.

When looking for critical points of $\A_k$ (and more generally 
of a given functional defined on a Hilbert manifold) it is natural to consider Palais-Smale sequences as a ``source of critical points'', since their limit points are by definition critical points. 
However, it is in general not true that Palais-Smale sequences have limit points. Therefore, it is worth looking for necessary and sufficient conditions for a Palais-Smale sequence to admit converging subsequences. 
Palais-Smale sequences with times going to zero surely do not possess limit points. However, they might occur only in connected components of $\mathcal M_Q$ that contain constant paths, as Lemma \ref{completeness0} shows. The next lemma ensures also that such Palais-Smale sequences may appear only at level zero. 

\begin{lemma}
Let $\gamma_h=(x_h,T_h)$ be a Palais-Smale sequence at level $c\in \R$ for $\A_k$ such that $T_h\rightarrow 0$. Then necessarily $c=0$.
\label{yesmodification}
\end{lemma}

\begin{proof}
First  we prove that
\begin{equation}
\int_0^{T_h} |\dot \gamma_h(t)|^2 \, dt = O(T_h) ,\quad \text{for} \ h\rightarrow +\infty .
\label{stimacont}
\end{equation}
Being $(x_h,T_h)$ a Palais-Smale sequence for $\A_k$, we have 
$$|d\A_k(x_h,T_h)| = o(1) ,\quad \text{for} \ h\rightarrow +\infty .$$
In particular, using (\ref{derivatainT}) we get that 
$$\left | d\A_k(x_h,T_h)\Big [\frac{\partial}{\partial T}\Big ] \right |= \left | \frac{\partial \A_k}{\partial T}(x_h,T_h) \right |= \left | \frac{1}{T_h} \int_0^{T_h} \Big [E\big (\gamma_h(t),\dot \gamma_h(t)\big ) - k \Big ]\, dt \right |=o(1)$$
and hence 
\begin{equation}
\alpha_h:=  \frac{1}{T_h}\int_0^{T_h} \Big [E\big (\gamma_h (t),\dot \gamma_h(t)\big ) -k\Big ]\, dt  \longrightarrow 0.
\label{convenergia}
\end{equation}

Now by assumption $E$ is quadratic at infinity and hence $E(q,v)\geq a' |v|^2 - b',$ for some $a'>0$ and $b'\in \R$. Using this in (\ref{convenergia}) we get that 
$$\alpha_h = \frac{1}{T_h}\int_0^{T_h} \Big [E\big (\gamma_h(t),\dot \gamma_h(t)\big )-k\Big ]\, dt \geq  \frac{1}{T_h} \int_0^{T_h} \Big [a'\, |\dot \gamma_h(t)|^2 - b'-k\Big ]\, dt$$
and hence 
$$\int_0^{T_h} |\dot \gamma_h(t)|^2 \, dt \leq \frac{T_h}{a'} \big [\alpha_h + b'+k\big ],$$
which implies (\ref{stimacont}). Since also $L$ is quadratic at infinity we have 
$$a |v|^2 - b \leq L(q,v) \leq \tilde{a} |v|^2 + \tilde{b}$$ 
for some constants $a,\tilde{a} >0$ and $b,\tilde{b}\in \R$. The first inequality implies 
\begin{align*}
\A_k(x_h,T_h) &=  \int_0^{T_h} \Big [ L(\gamma_h(t),\dot \gamma_h(t))+k \Big ]\, dt \\
                        &\geq a \int_0^{T_h} |\dot \gamma_h(t)|^2\, dt + T_h(k-b) = O(T_h)
                        \end{align*}
while the second yields
\begin{align*}
\A_k(x_h,T_h) &=   \int_0^{T_h} \Big [ L(\gamma_h(t),\dot \gamma_h(t))+k \Big ]\, dt \\
		       &\leq \tilde{a} \int_0^{T_h} |\dot \gamma_h(t)|^2\, dt + T_h(k+\tilde{b}) =O(T_h)
		       \end{align*}
and hence obviously $\A_k(x_h,T_h) \rightarrow 0$.
\end{proof}

\vspace{2mm}

The following lemma ensures the existence of converging subsequences for any Palais-Smale sequence with times bounded and bounded away from zero. The proof is analogous (with some minor adjustments) to the one of 
\cite[Proposition 3.12]{Con06} (or \cite[Lemma 5.3]{Abb13}), where the case of periodic orbits is considered; see \cite[Lemma 3.2.2]{Ass15} for the details. This lemma combined with Lemma \ref{yesmodification} above shows that the only Palais-Smale sequences at level $c\neq 0$ which may cause troubles are those for which the times diverge.

\begin{lemma}
Let $(x_h,T_h)$ be a Palais-Smale sequence at level $c\in \R$ for $\A_k$ in some connected component of $\mathcal M_Q$ with $\, 0<T_*\leq T_h\leq T^*<+\infty$. Then, $(x_h,T_h)$ is compact in $\mathcal M_Q$, meaning that it admits a converging subsequence. 
\label{lemmalimitatezza}
\end{lemma}


\section{The Ma\~n\'e critical value $c(L;Q_0,Q_1)$}
\label{themanecriticalvalueclq0q1}

The following numbers should be interpreted as energy levels and mark important dynamical and geometric changes for the Euler-Lagrange flow induced by the Tonelli Lagrangian $L$. The reader may take a look at 
\cite{Abb13} or \cite{CI99} for a survey on the relevance of these energy values and on their relation with the geometric and dynamical properties of the Euler-Lagrange flow. 
First, let us define the \textit{Ma\~n\'e critical value} associated to $L$ as 
\begin{equation}
c(L) := \inf \Big \{ k\in \R \ \Big |\ S_k(\gamma)\geq 0,\ \forall \gamma \text{ closed loop} \Big \}.
\label{c(L)}
\end{equation}
Second, we recall the definition of the \textit{Ma\~n\'e critical value of the Abelian cover} 
\begin{equation}
c_0(L) := \inf \Big \{ k\in \R \ \Big |\ S_k(\gamma)\geq 0 ,\ \forall \gamma \text{ closed loop homologous to zero} \Big \} .
\label{c0(L)}
\end{equation}

This is the relevant energy value, for instance, when trying to use methods coming from Finsler geometry. Indeed, for every $k>c_0(L)$ the Euler-Lagrange flow restricted to the energy level $E^{-1}(k)$ is conjugated to the geodesic flow defined by a suitable Finsler metric (see \cite{Abb13} for the details). 
On the other hand, for exact magnetic flows (i.e. Euler-Lagrange flows associated with magnetic Lagrangians) on surfaces, if $k<c_0(L)$ then there exist periodic orbits 
with energy $k$ which are local minimizers of the free-period Lagrangian action functional, as explained in 
 \cite{AMMP14} and in  \cite{CMP04}. One would be tempted to say that, at least on surfaces and for 
a suitable range of energies, a similar existence result of local minimizers for the free-time action functional $\A_k$ should hold also in our setting; this is unfortunately not the case, as we will see in Section \ref{counterexamples}.

When looking for periodic orbits, the energy value value which turns out to be relevant for the properties of the free-period action functional (see again \cite{Abb13} or \cite{Con06}) is however the so-called \textit{Ma\~n\'e critical value of the universal cover}
\begin{equation}
c_u(L) :=\inf \Big \{ k\in \R \ \Big |\ S_k(\gamma)\geq 0 ,\ \forall \gamma \text{ closed contractible loop} \Big \}.
\label{cu(L)}
\end{equation}
We also define 
\begin{equation}
e_0(L):= \max_{q\in M} E(q,0)
\label{e0(L)}
\end{equation}
to be the maximum of the energy on the zero section of $TM$. The topology of the energy level sets changes when crossing the value $e_0(L)$. In fact, for any $k>e_0(L)$, the energy level sets $E^{-1}(k)$ have all the same topology,
namely of a sphere bundle over $M$. This is instead false for $k<e_0(L)$, being the projection $E^{-1}(k) \rightarrow M$  not surjective any more. Notice that 
\begin{equation}
\min E  \leq e_0(L) \leq c_u(L) \leq c_0(L) \leq c(L) .
\label{relazionemane}
\end{equation}

If $L$ is a magnetic Lagrangian as in \eqref{magneticlag}, then $\min E=e_0(L)=0$. When the magnetic potential $\theta$ vanishes we additionally have
$$0= \min E = e_0(L)=c_u(L)=c_0(L) = c(L) ,$$ 
but in general the inequalities in (\ref{relazionemane}) are strict. See for instance \cite[Page 151]{Man97} (or Section \ref{counterexamples}) for an example where $e_0(L)< c_0(L)<c(L)$ and \cite{PP97} for an example where $c_u(L)<c_0(L)$. 
The values $c_u(L)$ and $c_0(L)$ clearly coincide when $\pi_1(M)$ is abelian; more generally, they coincide whenever $\pi_1(M)$ is amenable (cf. \cite{FM07}).

When the fundamental group of $M$ is rich, there are other Ma\~{n}\'e critical values, which are associated to the different covering spaces of $M$. We now show which one is relevant for our purposes. 
Given a covering map $p:M_1\rightarrow M$, consider the lifted Lagrangian
$$L_1 := dp\circ L:TM_1 \longrightarrow \R$$ 
and the associated critical value $c(L_1)$ as in (\ref{c(L)}). The following lemma is straightforward.

\begin{lemma}
There holds $c(L_1)\leq c(L)$. If $p$ is a finite covering, then $c(L_1)=c(L)$.
\label{coveringandmane}
\end{lemma}

\begin{oss}
Ma\~n\'e critical values have also an equivalent Hamiltonian definition (see for instance \cite{BP02}  or \cite{CI99}). As above let $M_1$ be a cover of $M$ and denote by $H_1$ the lift of the Tonelli Hamiltonian $H$ associated with $L$ to the cover $M_1$; then there holds 
\begin{equation}
c(L_1) =  \inf_{u\in C^\infty(M_1)} \sup_{q\in M_1}  H_1(q,d_qu) .
\label{hamiltoniancharacterization}
\end{equation}
\end{oss}

It is well known that regular covering spaces correspond to normal subgroups of $\pi_1(M)$, i.e. for any regular covering $p:M_1\rightarrow M$ there is a unique normal subgroup $G < \pi_1(M)$ with
$$M_1 \cong \bigslant{\widetilde{M}}{G}\, ,$$
where $\widetilde M$ denotes the universal cover of $M$. We denote the Ma\~{n}\'e critical value $c(L_1)$ of the lifted Lagrangian by
$$ c(L;G):=  c(L_1) .$$

\begin{lemma}
Let $G,G'< \pi_1(M)$ be two normal subgroups; then 
$$c(L; \langle G,G'\rangle ) = \max \big \{ c(L;G),c(L;G')\big \} ,$$
where $\langle G,G' \rangle$ denotes the (normal) subgroup generated by $G$ and $G'$.
\end{lemma}

\begin{proof}
Since $G < \langle G,G'\rangle$ is a normal subgroup, we have a covering 
$$p: \bigslant{\widetilde{M}}{G} \longrightarrow \bigslant{\widetilde{M}}{\langle G,G'\rangle}$$ 
and hence, by Lemma \ref{coveringandmane}, $c(L;G)\leq c(L;\langle G,G'\rangle)$. The same holds clearly also when considering $G'$ instead of $G$ and hence we get 
$$\max \big \{ c(L;G),c(L;G')\big \} \leq c(L;\langle G,G'\rangle ) .$$
Conversely,  let $k<c(L;\langle G,G'\rangle )$. By definition there exists 
$$\gamma =\alpha_1 \# \beta_1 \# ... \# \alpha_n \# \beta_n$$ 
with $\alpha_i\in G$, $\beta_i\in G'$ for all $i=1,...,n$, such that $S_k(\gamma)<0$. It follows
$$S_k(\gamma)=S_k(\alpha_1)+S_k(\beta_1) + ... +S_k(\alpha_n) + S_k(\beta_n) < 0 .$$ 

In particular there is one loop, say $\alpha_1$, such that $S_k(\alpha_1)<0$; hence, by definition we have  $k<c(L;G)$. This implies the opposite inequality.
\end{proof}

\vspace{4mm}

We want now to understand when the action functional $\A_k$ is bounded from below on each connected component of $\mathcal M_Q$. Thus, let $q_0\in Q_0$, $q_1\in Q_1$ and denote by
\begin{equation}
G_0 :=\ \big \langle \imath_*(\pi_1(Q_0,q_0))\big \rangle\, ,\ \ \ \ G_1 :=\ \big \langle \imath_*(\pi_1(Q_1,q_1))\big \rangle 
\end{equation}
the smallest normal subgroups in $\pi_1(M)$ which contain $\imath_*(\pi_1(Q_0))$, $\imath_*(\pi_1(Q_1))$ respectively, where $\imath:Q_0\hookrightarrow M$, $\imath:Q_1\rightarrow M$ are the inclusion maps.

Suppose that there exists a loop $\delta$ freely-homotopic to an element in $\imath_*(\pi_1(Q_0,q_0))$ and with $S_k(\delta)<0$. Under this assumption we want to show 
that $\A_k$ is unbounded from below on every connected component of $\mathcal M_Q$.
Without loss of generality we may assume that  $\delta \in \imath_*(\pi_1(Q_0,q_0))$, as otherwise we can choose any path $\nu$ from $q_0$ to $\delta(0)$,
$\eta\in \pi_1(M,Q_0)$ such that $\eta^{-1}\#\nu^{-1}\#\delta \#\nu\#\eta \in \imath_*(\pi_1(Q_0,q_0))$, and $n\in \N$ large enough such that 
$$S_k(\eta^{-1}\# \nu^{-1}\# \delta^n \# \nu\#\eta)  < 0.$$
Now fix $\sigma \in \mathcal M_Q$. Since $Q_0$ is connected there exists a path $\mu:[0,1]\rightarrow Q_0$ such that $(\mu^{-1}\#\sigma\# \mu) (0)=q_0$. Furthermore, we can choose $\mu$ in such a way that 
$$\A_k(\mu^{-1}\#\sigma\# \mu ) = S_k(\mu^{-1})+\A_k(\sigma)+S_k(\mu) \leq \A_k(\sigma) + c,$$
where $c$ is some constant independent of $\sigma$. Therefore, up to adding a uniformly bounded quantity to $\A_k(\sigma)$, we can assume without loss of generality that $\sigma(0)=q_0$. We claim that, for all $n\in\N$, the path $\sigma \#  \delta^n$ lies in the same connected component of $\sigma$. Indeed, there exists $\alpha \in \pi_1(Q_0,q_0)$ such that $\imath \circ \alpha$ is homotopic  to $\delta$ with 
base point $q_0$ fixed; in particular $\sigma \# \delta \sim \sigma \# (\imath \circ \alpha)$ and now it is easy to see that $\sigma \# (\imath \circ \alpha)\sim \sigma$. A homotopy
is for instance given by 
$$F:[0,1]\times [0,1]\rightarrow M,\quad F(s,\cdot ):= \sigma \# (\imath \circ \alpha)|_{[s,1]}.$$

Since $\A_k\big (\sigma \#  \delta^n \big )\rightarrow -\infty$ as $n\rightarrow +\infty$,  
we may conclude that, if such a loop $\delta$ exists then the free-time action functional $\A_k$ is unbounded from below on each connected component of $\mathcal M_Q$. 
In other words, $\A_k$ is unbounded from below on each connected component of $\mathcal M_Q$ if $k<c(L;G_0)$ (observe indeed that, by definition, 
for every $k<c(L;G_0)$ there exists a loop $\delta$ satisfying the requirements). Clearly the same 
holds when considering $Q_1$ instead of $Q_0$. Therefore we define the \textit{Ma\~n\'e critical value of the pair} $Q_0,Q_1$ as 
\begin{equation}
c(L;Q_0,Q_1) := c(L;\langle G_0,G_1\rangle) = \max \big \{c(L;G_0),c(L;G_1)\big \} .
\label{clh0h1}
\end{equation}

We can sum up the discussion above in the following

\begin{lemma}
For every $k< c(L;Q_0,Q_1)$, the free-time action functional $\A_k$ is unbounded from below on each connected component of $\mathcal M_Q$.
\label{unboundedness}
\end{lemma}

\begin{prop}
We have $k(L;Q_0,Q_1)\leq c(L;Q_0,Q_1)$.
\label{klminorecl}
\end{prop}
\begin{proof}
We prove this rigorously in case $L$ is a magnetic Lagrangian and then show how to adjust the proof in the general case. Recall that, if $L(q,v)=\frac 12 |v|_q^2 +\theta_q(v)$, then 
$$k(L;Q_0,Q_1)= \max \left \{\min_{q\in Q_0} \frac 12 |\mathcal P_0 w_q|^2, \min_{q\in Q_1} \frac 12 |\mathcal P_1 w_q|^2\right \},$$
where $\mathcal P_i:TM|_{Q_i}\rightarrow TQ_i$ is the orthogonal projection and $w_q\in T_qM$ is the unique vector representing $\theta_q \in T_q^*M$.
Suppose without loss of generality that 
$$k(L;Q_0,Q_1)=\min_{q\in Q_0} \frac 12 |\mathcal P_0 w_q|^2>0$$ 
and fix $k<k(L;Q_0,Q_1)$ (if $k(L;Q_0,Q_1)=0$ 
then there is nothing to prove, for we trivially have $c(L;Q_0,Q_1)\geq 0$). 
We now prove that there exists $\delta \in \imath_*(\pi_1(Q_0))$ with negative $(L+k)$-action; this yields that $k<c(L;G_0)$ and 
hence in particular $k<c(L;Q_0,Q_1)$, thus showing our claim. Consider a curve $u:[0,T]\rightarrow Q_0$ satisfying $\dot u(t)=-\mathcal P_0 w_{u(t)}$ 
for every $t\in [0,T]$ (observe that by assumption $\mathcal P_0w_q\neq 0$ for every $q\in Q_0$); a straightforward computation shows that 
$$S_k(u)=\int_0^T \left (k - \frac12 |\mathcal P_0 w_{u(t)}|^2 \right )\, dt \leq T \left (k - \min_{q\in Q_0} \frac 12 |\mathcal P_0 w_q|^2\right )\rightarrow -\infty$$
as $T\rightarrow +\infty$. Since $Q_0$ is compact and connected, for every $T>0$ we can find a path $\gamma_T:[0,1]\rightarrow Q_0$ connecting $u(T)$ with $u(0)$ 
and with $S_k(\gamma_T)$ uniformly bounded. It follows that, for $T$ large enough, $\delta:=\gamma_T\#u$ is a loop in $Q_0$ with negative $(L+k)$-action.

In the general case consider the restriction $H:N^*Q_0\rightarrow \R$; since $H$ is Tonelli, there exists a smooth section $q_0\mapsto \theta_{q_0}$
of the bundle $N^*Q_0\rightarrow Q_0$. Consider now 
$$u_{q_0}:= \frac{\partial H}{\partial p}(q_0,\theta_{q_0})\in T_{q_0}M.$$
Actually we have $u_{q_0}\in T_{q_0}Q_0$. Indeed if $p_0\in N_{q_0}^*Q_0$ then
$$p_0(u_{q_0})=\frac{\partial H}{\partial p}(q_0,\theta_{q_0})[p_0]=0.$$
By definition we have 
$$\min_{N^*Q_0} H = \min_{q_0\in Q_0} E(q_0,u_{q_0}).$$
The assertion follows now, as above, integrating the vector field $u$ on $Q_0$ long enough. 
\end{proof} 

The discussion above actually also proves that the connected components of $\mathcal M_Q$ correspond, with a slight abuse of notation, to the elements of $\pi_1(M,q_0)/\langle G_0,G_1\rangle$.

\begin{prop}
The connected components of $\mathcal M_Q$ correspond to the classes in $\pi_1(M,q_0)/\langle G_0,G_1\rangle$.
\end{prop}
\begin{proof}
Observe that $G_1$ is not a subgroup of $\pi_1(M,q_0)$, but it can be naturally identified with a subgroup of it. Consider 
$\gamma:[0,1]\rightarrow M$ such that $\gamma(0)=q_0$ and $\gamma(1)=q_1$, the induced isomorphism 
$$\psi:\pi_1(M,q_1)\rightarrow \pi_1(M,q_0),\quad \psi ([\alpha]):=[\gamma^{-1}\#\alpha \#\gamma],$$
and the subgroup $\psi(G_1)<\pi_1(M,q_0)$. Obviously, $\psi(G_1)$ is independent of the choice of the path $\gamma$.

Notice furthermore that the connected components of $\mathcal M_Q$ are in bijection with $\pi_0(\mathcal M_q)/\sim_{Q_0,Q_1}$, where $\mathcal M_q$ is the 
space of paths connecting $q_0$ with $q_1$ and $[u] \sim_{Q_0,Q_1}[v]$ if and only if there exist $g_0\in G_0$ and $g_1\in G_1$ such that $[u]=[g_1 \# v \# g_0]$.
The desired bijection is now given by 
$$\hspace{37mm} \pi_0(\mathcal M_q)/\sim_{Q_0,Q_1} \rightarrow \pi_1(M,q_0)/\langle G_0,G_1\rangle, \quad [u]\longmapsto [u\# \gamma].
\hspace{23mm} \qedhere$$
\end{proof}

\vspace{1mm}

\begin{oss}
If $Q_0\cap Q_1$ is not connected, there might be more than one connected component of $\mathcal M_Q$ containing constant paths as the following example shows.
Consider $\T^2$ as the square $[0,1]^2$ with identified sides and let $Q_0$ be a circle with center in $(\frac 12,\frac 12)$ and radius $r_0<\frac 12$. Let now
$Q_1$ be another circle with center in $(0,\frac 12)$ and radius $r_1<\frac 12$ such that $r_0+r_1>\frac 12$ and denote with $q_0,q_1,q_2,q_3$ the four intersection points (see Figure \ref{conncomp}). Clearly, the constant paths in $q_0$ and $q_1$ are contained in the same connected component of $\mathcal M_Q$; the same holds for the 
constant paths in $q_2$ and $q_3$. On the other hand, it is easy to see that the path 
$$a:[0,1]\rightarrow \T^2, \quad a(t)= ((q_0)_x + t, (q_0)_y)$$ is homotopic to the
constant path in $q_2$; in particular, the constant 
paths in $q_0$ and in $q_2$ are in different connected components of $\mathcal M_Q$, since $[a]\neq 0$ in $\pi_1(M,q_0)/\langle G_0,G_1\rangle=\pi_1(M,q_0)$.
\end{oss}

\begin{figure}[h]
\begin{center}
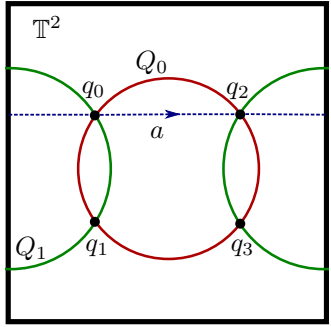
\end{center}
\caption{Constant paths may be contained in different connected components of $\mathcal M_Q$.}
\label{conncomp}
\end{figure}

We show now that, for $k\geq c(L;Q_0,Q_1)$, $\A_k$ is bounded from below on each connected component of $\mathcal M_Q$.
The proof is analogous to the one of \cite[Lemma 4.1]{Abb13}, where the case of periodic orbits is treated and $c(L;Q_0,Q_1)$ is replaced by $c_u(L)$.
Before proving this we need the following 

\begin{lemma}
Let $G<\pi_1(M)$ be a normal subgroup containing $G_0$ and let $p_G:M_G\rightarrow M$ be the corresponding 
covering map, where $M_G:=\widetilde M/G$ is endowed with the metric obtained by lifting the given metric on $M$. Then there exists $Q_G\subseteq M_G$ such that $Q_G\cong Q_0$ and 
$$p_G^{-1}(Q_0)=\bigcup_{\Gamma \in \pi_1(M)/G} \Gamma \cdot Q_G.$$
\label{lemmacovering}
\end{lemma}
\vspace{-6mm}
\begin{proof}
Denote with $\{U_j\}_{j\in J}$ the connected components of $p_G^{-1}(Q_0)$, so that 
$p_G^{-1}(Q_0)=\cup_{j\in J} U_j.$ Observe that $p_G|_{U_j}:U_j\rightarrow Q_0$ is injective for every $j\in J$. Indeed, consider $p_0,p_1 \in U_j$ projecting to the same point $q_0\in Q_0$
and pick any path $\alpha:[0,1]\rightarrow U_j$ such that $\alpha(0)=p_0$ and $\alpha(1)=p_1$. Then $[p_G\circ \alpha ]\in \imath_*(\pi_1(Q_0))$; this implies in particular that $[p_G\circ \alpha]$ is a non trivial Deck-transformation, which is clearly impossible since $G_0\subseteq G$.
It follows that $p_G|_{U_j}:U_j\rightarrow Q_0$ is a homeomorphism for every $j\in J$; in particular, $U_j$ is compact for every $j\in J$.

Fix $Q_G:=U_j$ for some $j\in J$. Clearly, for every $\Gamma \in \pi_1(M)/G$, $\Gamma \cdot Q_G$ is a connected component of $p_G^{-1}(Q_0)$. Conversely, 
let $U_G$ be a connected component of $p_G^{-1}(Q_0)$. Consider $q_0\in Q_0$ and its preimages $q\in Q_G$ and $u\in U_G$ under $p_G$. 
Since the cover is normal, the deck group $\pi_1(M)/G$ acts transitively on $p_G^{-1}(q_0)$; therefore, there exists $\Gamma\in \pi_1(M)/G$ such that 
$\Gamma(q)=u$ and hence $U_G\cap \Gamma \cdot Q_G\neq \emptyset$. It follows that $U_G=\Gamma \cdot Q_G$.
\end{proof}

\begin{lemma}
For every $k\geq c(L;Q_0,Q_1)$ the free-time action functional $\A_k$ is bounded from below on every connected component of $\mathcal M_Q$.
\label{boundedness}
\end{lemma}

\begin{proof}
Consider $\sigma:[0,T]\rightarrow M$ in some connected component of $\mathcal M_Q$ and
\begin{equation}
M_1:=\widetilde{M}/\langle G_0,G_1\rangle \stackrel{p}{\longrightarrow}  M,
\label{rivestimento}
\end{equation}
where $\widetilde{M}$ is the universal cover. Denote by $\sigma_1$ the lift of $\sigma$ to $M_1$; we lift the metric of $M$ to $M_1$ 
and notice that, by Lemma \ref{lemmacovering}, having fixed the connected component of $\mathcal M_Q$, we have that
$\text{dist} (\sigma_1(0),\sigma_1(T))$
is uniformly bounded. Therefore, there exists a path $\eta_1:[0,1]\rightarrow M_1$ which joins $\sigma_1(T)$ with $\sigma_1(0)$ and has uniformly bounded action 
$$\tilde{S}_k(\eta_1)= \int_0^1 \Big [L_1(\eta_1(t),\dot \eta_1(t))+k \Big ]\, dt  \leq  C ,$$
where $L_1$ denotes the lifted Lagrangian on $M_1$. If $\eta:=p \circ \eta_1$, then the juxtaposition $\sigma \# \eta\in \langle G_0,G_1\rangle$ and, since 
by assumption $k\geq c(L;Q_0,Q_1)$, we get 
$$0 \leq S_k(\sigma \# \eta)= \A_k(\sigma) + S_k(\eta)=\A_k(\sigma) + \tilde{S}_k(\eta_1)\leq \A_k(\sigma) + C$$
from which we deduce that $\A_k(\sigma) \geq -C$.
\end{proof}

\begin{cor}
If $k>c(L;Q_0,Q_1)$, then every Palais-Smale sequence for $\A_k$ in a connected component $\mathcal N$ of $\mathcal M_Q$ that does not contain constant paths admits a converging subsequence. The same holds if $\mathcal N$ contains constant paths, provided that the Palais-Smale sequence is at level $c\neq 0$.
\label{compattezza}
\end{cor}

\begin{proof} 
Under the assumptions of the corollary we know by Lemmas \ref{completeness0} and \ref{yesmodification} that the times $T_h$ are bounded away from zero. 
Therefore, in virtue of Lemma \ref{lemmalimitatezza} it is enough to show that the $T_h$'s are uniformly bounded from above. Since 
$$\A_k(x,T) = \A_{c(L;Q_0,Q_1)} (x,T)+\Big (k-c(L;Q_0,Q_1)\Big )T$$
for any $(x,T)\in \mathcal M_Q$, the period 
$$T_h =\frac{1}{k-c(L;Q_0,Q_1 )} \Big [\A_k(x_h,T_h)-\A_{c(L;Q_0,Q_1)}(x_h,T_h)\Big ]$$ 
is clearly uniformly bounded from above, being $\A_k$ bounded on the Palais-Smale sequence and being $\A_{c(L;Q_0,Q_1)}(x_h,T_h)$ bounded from below 
by Lemma \ref{boundedness}.  
\end{proof}

\vspace{1mm}

When looking for connecting orbits in case $Q_0,Q_1$ intersect, there is another relevant energy value which we now define. 
In the next section we will namely use Corollary \ref{compattezza} to construct orbits satisfying the conormal boundary conditions as action minimizers. 
However, when minimizing on a  connected component containing constant paths we need to ensure that the infimum is negative (observe that such an infimum cannot be positive). This is not always the case as the example in the introduction shows. 
Thus, let $\mathcal N$ be a connected component of $\mathcal M_Q$ containing constant paths and define the energy value
\begin{equation}
k_{\mathcal N}(L) :=\inf \Big \{k\in \R \ \Big |\ \A_k(\gamma)\geq 0, \ \forall \ \gamma \in \mathcal N\Big \} .
\label{komega(L)}
\end{equation}

By definition we readily see that $c(L;Q_0,Q_1 )\leq k_{\mathcal N}(L)$. In the next section we show that in the interval 
$(c(L;Q_0,Q_1 ),k_{\mathcal N}(L))$ we find Euler-Lagrange orbits in $\mathcal N$ satisfying the conormal boundary conditions by minimizing $\A_k$ on $\mathcal N$. 
Notice that the considered interval might be empty but in general it is not; see Section \ref{counterexamples} for an example.
Existence results above $k_{\mathcal N}(L)$ are in general achievable only under additional assumptions  (c.f. Theorem \ref{teosupercriticalintro}). 

Another ``natural'' energy value is given by 
$$k_0(L) :=  \inf \Big \{k\in \R \ \Big |\ \A_k(\gamma)\geq 0 , \ \forall \ \gamma \in \mathcal M_Q\Big \} .$$

It is interesting to study the relation between $k_0(L)$ and the critical value $c(L;Q_0,Q_1)$ and, more generally, the other critical values we introduced in this section; this will also give us an estimate on how much the 
various critical values can differ. 

Clearly $c(L;Q_0,Q_1) \leq k_0(L)$. 
We claim that actually $c(L) \leq  k_0(L).$ Thus, consider $k<c(L)$; by definition there exists a loop $\delta$ such that $S_k(\delta)<0$. It is now easy to construct a path from $Q_0$ to $Q_1$ with negative action: 
pick any path $\eta$ from a point $q_0\in Q_0$ to the base point $\delta(0)$, then wind $n$-times around $\delta$ and finally join $\delta(0)$ with a point 
$q_1\in Q_1$ by a path $\mu$. If $n$ is large enough then
$$\A_k(\mu\# \delta^n \# \eta) = S_k (\mu) + n S_k(\delta) + S_k(\eta) < 0 ,$$
which implies  $k<k_0(L)$ and the claim follows. Therefore, we have 
$$e_0(L) \leq c_u(L)\leq c(L;Q_0,Q_1) \leq c(L) \leq k_0(L),$$
where the second and third inequalities follow from Lemma \ref{coveringandmane}. It is easy to see that in general there is no relation between $c_0(L)$ and $c(L;Q_0,Q_1)$.

In order to estimate how much the various Ma\~n\'e critical values can differ, one can measure the difference $k_0(L)-e_0(L)$. Thus, consider the smooth one-form 
$$\theta(q)[v]:= d_vL(q,0)[v] ;$$
by taking a Taylor expansion and by using (\ref{secondinequality}), we get that
\begin{align*}
L(q,v)&= L(q,0)+d_vL(q,0)[v] + \frac{1}{2}\, d_{vv}L(q,sv)[v,v] \\ 
          &\geq - E(q,0)+ \theta(q)[v] + a |v|^2,
\end{align*}
where $s\in [0,1]$ is a suitable number. If we set $\gamma(t):=x(t/T)$, then we obtain
\begin{align*}
\A_k(x,T) = \A_k(\gamma) &\geq \int_0^T \Big [-E(\gamma(t),0)+\theta(\gamma(t))[\dot \gamma(t)] + a |\dot \gamma(t)|^2 + k \Big ]\, dt \\ 
   &= \int_0^T \Big [k-E(\gamma(t),0)\Big ]\, dt + \int_0^T \gamma^*\theta + a \int_0^T |\dot \gamma(t)|^2 \, dt\\
   &\geq \Big [k-e_0(L)\Big ] T  + \frac{a}{T} l(\gamma)^2 -  \|\theta\|_\infty  l(\gamma) .
\end{align*}
For $k>e_0(L)$ and $T$ fixed, the latter expression is a parabola in $l(\gamma)$ with minimum
$$(k-e_0(L))  T  -  \frac{\|\theta\|_\infty^2}{4a}T= \left (k-e_0(L) - \frac{\|\theta\|_\infty^2}{4a}\right ) T.$$
In particular, if 
$$k > e_0(L) + \frac{\|\theta\|_\infty^2}{4a} ,$$ 
then $\A_k(\gamma)\geq 0$ for any path $\gamma$ connecting $Q_0$ with $Q_1$ and this implies, by the definition of $k_0(L)$, that $k>k_0(L)$. Therefore we get
$$k_0(L)  \leq e_0(L) + \frac{\|\theta\|_\infty^2}{4a} .$$ 
We sum up the discussion above with the following

\begin{prop}
Let $L:TM\rightarrow \R$ be a Tonelli Lagrangian, $Q_0,Q_1\subseteq M$ closed submanifolds. Then the following chain of inequalities holds 
\begin{equation}
e_0(L) \leq c_u(L) \leq c(L;Q_0,Q_1) \leq c(L) \leq k_0(L) \leq e_0(L) + \frac{\|\theta\|_\infty^2}{4a}.
\label{chainofcriticalvalues}
\end{equation}
In particular, if $\theta\equiv 0$, that is if $L$ is a mechanic Lagrangian as in \eqref{mechaniclag}, we retrieve 
$$e_0(L)= c_u(L)=c(L;Q_0,Q_1)= c(L)= k_0(L).$$
\end{prop}


\section{Existence results for high energies}
\label{existenceresultsforhighenergies}

In this section, building on the analytical backgrounds introduced in the previous ones, we prove Theorem \ref{teosupercriticalintro}. 
However, before proving the Theorem we need some preliminaries. 

Thus, suppose that $Q_0,Q_1\subseteq M$ are two closed connected intersecting submanifolds and let $\mathcal N$ be the connected component of 
$\mathcal M_Q$ containing the constant paths. Assume in addition that $Q_0\cap Q_1$ is connected.
We say that $\mathcal N$ has the \textit{retraction property} if there exists a neighborhood $\mathcal U$ of $Q_0\cap Q_1$ in $\mathcal N$ such that $Q_0\cap Q_1$ is a strong deformation retract of $\mathcal U$. 

\begin{oss}
$\mathcal N$ has the retraction property, for instance, if
$Q_0\cap Q_1=\{q\}$ or if  there exists $\epsilon >0$ small enough such that there are no geodesics with length $0<\ell<\epsilon$ from $Q_0$ to $Q_1$ and hitting $Q_0$ and $Q_1$ orthogonally. 
In the first case indeed, if $B_i\subseteq Q_i$ is a contractible neighborhood of $q$ in $Q_i$, for $i=0,1$, then $\mathcal U$ retracts on the image of the map $B_0\times B_1\hookrightarrow \mathcal U$ which associates to every $(q_0,q_1)$ the shortest geodesic connecting $q_0$ with $q_1$. Since $B_0\times B_1$ retracts on $(q,q)$, 
the conclusion follows.

If the second property holds then we can use the negative gradient flow of the kinetic energy 
$$e:H^1_Q([0,1],M)\rightarrow \R,\quad e(x)=\int_0^1|x'(s)|^2\, ds$$
to deform the set $\{(x,T)\in \mathcal N\ |\ e(x)<\epsilon^2\}$ into the set of constant paths. 
Notice that this latter property holds for $Q_0=Q_1$ and for generic choice of $Q_0$ and $Q_1$.
\end{oss}

\begin{proof}[Proof of Theorem \ref{teosupercriticalintro}.] Suppose $\mathcal N'$ is a connected component of $\mathcal M_Q$ not containing constant paths. 
Lemma \ref{completeness0} implies that the sublevels of $\A_k$ in $\mathcal N'$
$$\big \{(x,T)\in \mathcal N'\ \big | \A_k(x,T)\leq c\big \}$$
are complete. Moreover, Corollary \ref{compattezza} implies that $\A_k$ satisfies the Palais-Smale condition on $\mathcal N'$ for every $k>c(L;Q_0,Q_1)$. 
We may then conclude that $\A_k$ has a global minimizer on $\mathcal N'$ by taking a minimizing sequence for $\A_k$ as Palais-Smale sequence.

We prove now statement 2,~(a). Thus, let $\mathcal N$ be the connected component of $\mathcal M_Q$ containing the constant paths. Consider first $k\in (c(L;Q_0,Q_1), k_{\mathcal N}(L))$;
since $c:= \inf \A_k <0\, $, the sublevel sets of $\A_k$ in $\mathcal N$
$$\Big\{(x,T)\in \mathcal N \ \Big |\ \A_k(x,T) \leq c+\epsilon \Big \}$$ 
are complete for every $\epsilon >0$ small by Corollary \ref{cornegative}.  Moreover, Lemma \ref{yesmodification} implies that all the Palais-Smale sequences at level $c$ have $T_h$'s bounded away from zero and hence
$\A_k$ satisfies the Palais-Smale condition at level $c$ by Lemma \ref{boundedness}. We now retrieve the existence of a  global minimizer for $\A_k$ in $\mathcal N$ exactly as above.

Suppose now that $Q_0\cap Q_1$ is connected, $\mathcal N$ has the retraction property, and there exists $l\geq 1$ such that $\pi_l(\mathcal N,Q_0\cap Q_1)\neq 0$. Fix $k>k_{\mathcal N}(L)$; in this case, we retrieve the desired Euler-Lagrange orbit using a minimax argument 
analogous to that used by Lusternik and Fet \cite{FL51} in their proof of the existence of one closed geodesic on a simply connected manifold (see also \cite{Abb13} or \cite{Con06} for an application to the existence of periodic Euler-Lagrange orbits; in that case
$k_{\mathcal N}(L)$ is replaced by $c_u(L)$). 
By assumption there exists a non-trivial element $\mathcal H \in  \pi_l(\mathcal N,Q_0\cap Q_1)$ and therefore we can consider the minimax value 
$$c:= \inf_{h\in \mathcal H} \max_{\zeta \in D^{l}} \A_k(h(\zeta)).$$ 

Let us show that $c>0$; since $\mathcal H$ is non-trivial, there exists a positive number $\lambda$ such that for every map $h=(x,T):(D^{l},S^{l-1})\rightarrow (\mathcal N,Q_0\cap Q_1)$ belonging to the class $\mathcal H$ there holds 
$$\max_{\zeta\in D^{l}} \ l(x(\zeta))\geq \lambda,$$ 
where as usual $l(x(\zeta))$ denotes the length of the path $x(\zeta)$. This follows simply from the fact that $\displaystyle \inf_{\gamma \in \partial \mathcal U} l(\gamma)>0$. If $(x,T)\in \mathcal N$ has length $l(x)\geq \lambda$, then (\ref{firstinequality}) implies that 
\begin{align*}
\A_k(x,T) &= T\int_0^1 \Big [L\Big ( x(s),\frac{x'(s)}{T}\Big ) + k \Big ]\, ds \\
                &\geq \frac{a}{T}\int_0^1 |x'(s)|^2\, ds +  T(k-b)\\
                &\geq \frac{a}{T}  l(x)^2+ T(k-b)\\
                &\geq \frac{a}{T} \lambda^2 +T(k-b) .
\end{align*}

Since $\lambda >0$, the above inequality implies that if $(x,T)\in \mathcal N$ has length $l(x)\geq \lambda$ and action $\A_k(x,T)\leq c+1$ then 
$$c+1 \geq \frac{a}{T} \lambda^2 +T(k-b)$$ 
and hence $T\geq T_0$ for some $T_0>0$, because the quantity on the righthand-side goes to infinity as $T\rightarrow 0$. Now let $h\in \mathcal H$ be such that 
$$\max_{\zeta \in D^{l}}\A_k(h(\zeta)) \leq c+1 ;$$
then by the above considerations there exists $(x,T)\in h(D^{l})$ with $T\geq T_0$ and 
$$\A_k(x,T) =  \A_{k_{\mathcal N}(L)} (x,T) + \big (k-k_{\mathcal N}(L)\big )T \geq \big (k-k_{\mathcal N}(L) \big ) T_0 >0 .$$

The argument above shows that the minimax value $c$ is strictly positive. The celebrated minimax theorem (see for instance \cite[Theorem 1.8 and Remark 1.11]{Abb13} or \cite[Theorem 2.5.3]{Ass15}), 
combined with Lemma \ref{convergenza}, ensures the existence of a Palais-Smale sequence at level $c$. 
Since $c>0$ we also get from Lemma \ref{yesmodification} that the $T_h$'s are bounded away from zero, so that by Corollary \ref{compattezza} the Palais-Smale sequence has a limit point in $\mathcal N$, which gives us the 
required connecting orbit, thus proving 2,~(b).

Suppose now that $Q_0\cap Q_1$ is not connected and at least one connected component $\Omega$ of $Q_0\cap Q_1$ is isolated; in this case it is easy to see that there are always non-trivial classes in $\pi_1(\mathcal N,Q_0\cap Q_1)$. Indeed, any continuous map $u:(D^1,S^0)\rightarrow (\mathcal N, Q_0\cap Q_1)$ such that $u(0)\in \Omega$ and $u(1)\in \Omega'$ for some $\Omega'\neq \Omega$ represents a non-trivial class in $\pi_1(\mathcal N,Q_0\cap Q_1)$. Moreover, for every such a continuous map $u$ we have 
$$\max_{s\in [0,1]} l(u(s))\geq \inf_{\partial \mathcal U} l>0,$$
where $\mathcal U$ is an open set such that $\Omega\subseteq \mathcal U$ and $\Omega'\cap \overline{\mathcal U}=\emptyset$. The proof of 2,~(c) follows now repeating the argument used to prove 2,~(b). 
\end{proof}

\begin{proof}[Proof of Corollary \ref{arnoldchordsupercriticalintro}] The first statement and the first part of the second one follow trivially from the corresponding statements of Theorem \ref{teosupercriticalintro}. The second part of the second statement follows from Theorem \ref{teosupercriticalintro} and the fact that $\pi_l(\mathcal N,Q_0)\cong \pi_{l+1}(M,Q_0)$, for every $l\geq 0$. 

Thus it remains to show that for every $k>c(L;Q_0)$, $k\neq k_{\mathcal N}(L)$, there always exists an Arnold chord if $Q_0\neq M$. By 
connectedness of the spaces we have $\pi_0(M,Q_0)=\{0\}$. Moreover, if $\mathcal N$ is the only connected component of $\mathcal M_Q$, then 
we also have $\pi_1(M,Q_0)=\{0\}$. Therefore, if $\mathcal N$ is the only connected component of $\mathcal M_Q$ and $\pi_l(\mathcal N,Q_0)=\{0\}$ 
for every $l\geq 1$, then $\pi_l(M,Q_0)=\{0\}$ for every $l\geq 0$. 
But then the Hurewicz theorem (cf. \cite[Theorem 4.37]{Hat02}) 
would yield $H_l(M,Q_0)=0$ for every $l\in \N$ and, hence, $H_l(M)\cong H_l(Q_0)$, which is a contradiction for $l=\dim M$.
\end{proof}


\section{Existence results for low energies}
\label{existenceresultsforlowenergies}

In this section we study the existence of Euler-Lagrange orbits satisfying the conormal boundary conditions for subcritical energies $k<c(L;Q_0,Q_1)$.  As already explained in the introduction, this problem is harder than the corresponding one for supercritical energies. 

Throughout this section we assume that the submanifolds $Q_0$ and $Q_1$ intersect; we will get back to the case $Q_0\cap Q_1=\emptyset$ in Section \ref{counterexamples} 
showing that, under this assumption, the first statement of Theorem \ref{teosupercriticalintro} is optimal.

We start by considering the following particular case, which should help to understand the general situation later on. Let $Q_0$ and $Q_1$ be two closed connected submanifolds which intersect in one point, say $p$.
We show that, under the assumption (\ref{condizionetheta}) on the Lagrangian $L$, for every $k\in (E(p,0),c(L;Q_0,Q_1))$ the action functional $\A_k$ exhibits a mountain-pass geometry on the connected component $\mathcal N$ of $\mathcal M_Q$ 
that contains the constant paths. 

Since for any $k\in (E(p,0),c(L;Q_0,Q_1))$ the free-time action functional $\A_k$ is unbounded from below, it makes sense to define the following class of paths in $\mathcal N$
\begin{equation}
\Gamma := \Big \{ u :[0,1]\rightarrow \mathcal N \ \Big |\ u(0)=(p,T) ,\ T\leq T_0,\  \A_k(u(1))<0 \Big \}  ,
\label{gammasemplice}
\end{equation}
where $T_0>0$ is chosen so small that the class $\Gamma$ is invariant under (a suitable truncation of) the negative gradient flow of $\A_k$. The existence of $T_0$ will be shown  in Remark \ref{osstruncation}.
Define now 
\begin{equation}
\theta_q(\cdot) := d_v L (q,0)[\cdot], \quad \forall q\in M ,
\label{definizionetheta}
\end{equation}
and assume that there exists an open neighborhood $\mathcal U$ of $p$ such that
\begin{equation}
\theta_q \equiv 0, \quad \forall  q \in \mathcal U .
\label{condizionetheta}
\end{equation}

Without loss of generality we may suppose that $\mathcal U=B_r$ is an open ball with radius $r$ around $p$. Under the assumption (\ref{condizionetheta}) we show the desired mountain-pass geometry 
for the action functional $\A_k$. Namely, we prove that there is $\alpha>0$ such that 
$$\max_{s\in[0,1]} \A_k(u(s)) \geq \alpha , \quad \forall u\in \Gamma.$$

Here is the scheme of the proof: we first show that if the length of a path $\gamma$ connecting $Q_0$ and $Q_1$ is sufficiently small then the action of $\gamma$ needs to be non-negative. 
Therefore, for every element $u\in \Gamma$ there must be an $s\in [0,1]$ such that $l(u(s))=\epsilon$ for a suitable $\epsilon>0$. Now we get the assertion showing that every path with length 
$\epsilon$  has $\A_k$-action bounded away from zero by a positive constant.

Since $Q_0$ and $Q_1$ intersect only in $p$, for every $\delta>0$ there exists $\lambda_\delta>0$ such that  
$$d(Q_0\setminus B_\delta, Q_1\setminus B_\delta) \geq \lambda_\delta ,$$
where $B_\delta$ denotes the ball with radius $\delta$ around $p$. In other words, every path connecting $Q_0$ to $Q_1$ with starting and ending point outside $B_\delta$ has length larger than $\lambda_\delta$. It is clear now that, if $\epsilon>0$ is sufficiently small,
then every path $\gamma$ connecting $Q_0$ to $Q_1$ with length $l(\gamma)\leq \epsilon$ is entirely contained in $\mathcal U=B_r$. Indeed, fix $\delta<r$. If $\epsilon <\min \{r-\delta,\lambda_\delta\}$, then at least one between the starting and ending point of $\gamma$ is contained in 
$B_\delta$, say $\gamma(0)\in B_\delta$, and 
\begin{equation}
d(\gamma(t),p) < d(\gamma(t),\gamma(0))+ d(\gamma(0),p) < \epsilon + \delta < r , \quad \forall  t .
\label{gammainu}
\end{equation}
A Taylor expansion together with the bound (\ref{secondinequality}) implies
\begin{align}
L(q,v) &= L(q,0) + d_vL(q,0)[v] + \frac12 d_{vv} L(q,sv)[v,v] \nonumber \\ 
           &\geq - E(q,0) + \theta_q(v) + a |v|^2 .
\label{stimacontheta}
\end{align}

Let now $k>E(p,0)$. Up to choosing a smaller neighborhood $\mathcal U$ of $p$ (thus, a smaller $\epsilon$), the continuity of the energy implies that 
$$k > \sup_{q \in \mathcal U} E(q , 0 ) .$$
Using (\ref{condizionetheta}), (\ref{gammainu}) and (\ref{stimacontheta}) we now compute for every $\gamma=(x,T)$ with length $l(x)\leq \epsilon$ 
\begin{align*}
\A_k(x,T) &\geq \int_0^T \Big [-E(\gamma(t),0) + \theta_{\gamma(t)}(\gamma'(t)) + a \, |\gamma'(t)|^2 + k \Big ]\, dt \\ 
                &\geq T \big (k-E(p,0)\big ) + \frac{a}{T} \, l(x)^2 
\end{align*}
which is a non-negative quantity. It follows that for every $u\in \Gamma$ there is $s\in[0,1]$ such that $l(u(s))=\epsilon$; for such $s$ we obtain 
\begin{align*}
\A_k(u(s)) &\geq T \big (k-e_0(L)\big ) + \frac{a}{T} \epsilon^2 \\
&\geq 2\epsilon \sqrt{a(k-e_0(L))} \\
& =: \alpha
\end{align*}
as we wished to prove. 

\begin{oss}
Since $\A_k(p,T) = T  \big (k-E(p,0)\big ) > 0 $ goes to zero as $T\rightarrow 0$, we can choose $T_0$ in the definition of $\Gamma$ such that $\A_k(p,T_0)\leq \alpha/4$. In this 
way the class $\Gamma$ becomes invariant under the negative gradient flow of $\A_k$ truncated below level $\alpha/2$.
\label{osstruncation}
\end{oss}

We are now ready to deal with the general case. Let $Q_0,Q_1\subseteq M$ be closed and connected submanifolds with non-empty intersection. Let $\Omega\subseteq Q_0\cap Q_1$ be an isolated connected component of $Q_0\cap Q_1$ and denote with $\mathcal N$ the connected component of $\mathcal M_Q$ containing the constant paths. Define 
\begin{equation}
k_{\Omega} := \min \left \{c(L;Q_0,Q_1),\, \max_{q\in \Omega} E(q,0) + \max_{q\in \Omega} \frac{|\theta_q|^2}{4a}\right \},
\label{kq0q1}
\end{equation}
where $\theta_q$ is as in (\ref{definizionetheta}), $|\cdot|$ is the dual norm on $T^*M$ induced by the Riemannian metric on $M$ and $a>0$ is such that (\ref{secondinequality}) is satisfied. Now set the minimax class
\begin{equation}
\Gamma := \Big \{u=(x,T) :[0,1]\rightarrow \mathcal N \ \Big |\ x(0)=p\in \Omega, \ T(0)\leq T_0, \ \A_k(u(1))<0 \Big \} ,
\label{Gamma}
\end{equation}
where as above $T_0$ will be chosen such that $\A_k(p,T_0)\leq \alpha/4$ for every $p\in \Omega$, where 
$\alpha>0$ will be the constant given by Lemma \ref{lemmaminimax1}. In this way $\Gamma$ will become invariant under the negative gradient flow of $\A_k$
truncated below level $\alpha/2$.

Lemma \ref{lemmaminimax1} states that, for every $k \in (k_{\Omega},c(L;Q_0,Q_1))$, the action functional $\A_k$ has a mountain-pass geometry on 
$\mathcal N$, where the two valleys are represented by the constant paths (in $\Omega$) and by the paths with negative action. 
Notice however, that the considered interval could be empty; this happens, for instance, when 
$$\max_{q\in \Omega} E(q,0) = e_0(L),\quad \max_{q\in \Omega} \frac{|\theta_q|^2}{4a} = \frac{\|\theta\|_\infty^2}{4a} ,$$
as the chain of inequalities (\ref{chainofcriticalvalues}) shows. However, this is not always the case as we will show in the counterexample section. 
Observe also that, when $Q_0\cap Q_1=\{p\}$ with $\theta_p=0$, the energy value $k_{Q_0\cap Q_1}$ reduces to the above considered $E(p,0)$.
 
\begin{lemma}
Let $Q_0,Q_1\subseteq M$ be two closed connected submanifolds with non-empty intersection, let $\Omega$ be an isolated connected component of $Q_0\cap Q_1$ and let $k_{\Omega}$ be as  defined in (\ref{kq0q1}). For every $k\in (k_{\Omega},c(L;Q_0,Q_1))$ there exists 
$\alpha>0$ such that 
$$\inf_{u\in \Gamma} \max_{s\in [0,1]} \A_k(u(s)) \geq \alpha.$$
\label{lemmaminimax1}
\end{lemma}

\vspace{-6mm}

\begin{proof}
The proof follows from the one in the particular case treated above with minor adjustments. Consider a neighborhood $\mathcal U$ of $\Omega$ such that 
\begin{equation}
k > \sup_{q\in \mathcal U} E(q,0) + \sup_{q\in \mathcal U} \frac{|\theta_q|^2}{4a}.
\label{kmaggiorecectheta}
\end{equation}

As in the particular case one shows now that, if $\epsilon >0$ is sufficiently small, then every path joining $Q_0$ to $Q_1$ with length less than or equal to $\epsilon$ and contained 
in the connected component of $\{(x,T)\in \mathcal N \ |\ l(x)\leq \epsilon\}$ containing $\Omega$ has image contained in $\mathcal U$ (here is where we need the assumption 
$\Omega$ isolated). 
Pick now such an $\epsilon$; using (\ref{stimacontheta}) we compute for every $\gamma=(x,T)$ with $l(x)\leq \epsilon$
\begin{align*}
\A_k(x,T) &\geq \int_0^T \Big [-E(\gamma(t),0) + \theta_{\gamma(t)}(\gamma'(t)) + a |\gamma'(t)|^2 + k \Big ]\, dt  \\ 
                 &\geq \left (k - \sup_{q\in \mathcal U} \ E(q,0) \right ) T + \frac{a}{T}\, l(x)^2 + \int_0^T \theta_{\gamma(t)}(\gamma'(t))\, dt \\ 
                 &\geq \left (k - \sup_{q\in \mathcal U} \ E(q,0) \right ) T+ \frac{a}{T}\, l(x)^2 - \left (\sup_{q\in \mathcal U} |\theta_q|\right ) l(x) .
\end{align*}
To ease the notation let us define 
$$c_E :=  \sup_{q\in \mathcal U} E(q,0),  \quad c_\theta := \sup_{q\in \mathcal U} |\theta_q|$$ 
and consider the function of two variables 
$$f:(0,+\infty)\times [0,\epsilon] \rightarrow \R ,\quad f(T, l) := \big (k - c_E\big ) T + \frac{a}{T}  l^2 - c_\theta  l.$$
For every $l$ fixed the function $f$ has minimum 
$$\left ( 2\sqrt{a(k-c_E)} - c_\theta\right ) l $$
and this quantity is positive if and only if (\ref{kmaggiorecectheta}) is satisfied. Now, arguing as above, we get that for every $u\in \Gamma$ there exists $s\in [0,1]$ such that $l(u(s))=\epsilon$. For this $s$ we readily have
$$\A_k(u(s)) \geq \left ( 2\sqrt{a(k-c_E)} - c_\theta\right ) \epsilon =: \alpha> 0 ,$$
exactly as we wished to prove.
\end{proof}

\begin{oss}
When $\Omega$ consists of more than one point one would be tempted to replace in the definition of $k_{\Omega}$ the maximum of the energy on $\Omega$ with the corresponding minimum, hence defining 
\begin{equation}
k_{\Omega}^- := \min \left \{c(L;Q_0,Q_1), \, \min_{q\in \Omega} E(q,0)  + \max_{q\in \Omega} \frac{|\theta_q|^2}{4a}\right \},
\end{equation}
and show that the conclusion of Lemma \ref{lemmaminimax1} holds even considering the a priori larger interval $(k_{\Omega}^-,c(L;Q_0,Q_1))$. This is however not the case, since under these assumptions there are constant paths with negative $\A_k$-action. However, it seems reasonable to us that an argument analogous to the one in \cite{Abb13}, where the case of periodic orbits is considered and $k_{\Omega}^-$, $k_{\Omega}$ are replaced by $\min E$, $e_0(L)$ respectively, should go through in this setting, at least under some 
mild additional assumptions (such as $\Omega$ being a CW-complex). Namely, in the energy range $(k_{\Omega}^-,k_{\Omega})$, instead of the class $\Gamma$, one should 
consider the class of deformations 
$u=(x,T):[0,1]\times \Omega \rightarrow \mathcal N$
of the space of constant paths into the space of paths with negative $\A_k$-action
\begin{equation*}
\Gamma_{\Omega} := \Big \{u=(x,T) \ \Big | \ x(0,q) = q , \ \A_k(u(1,q))< 0, \ \forall q\in \Omega\Big \} .
\end{equation*}
However, it is a priori not clear why the class $\Gamma_\Omega$ should be non-empty. Indeed, in order to show that the corresponding class in the periodic setting is non-empty, besides the CW-complex structure of $M$ one has to use the iteration of loops and Bangert's trick of pulling one loop at a time \cite{Ban80} (we refer again to \cite{Abb13} and references therein for the details), which do not immediately generalize to our setting. 
This will be subject of future research.  
\end{oss}
%

If $\Omega$ is not isolated then the proof of Lemma \ref{lemmaminimax1} might fail. Nevertheless, we can define the energy value $k_\Omega$ in a suitable fashion so that
the conclusion of Lemma \ref{lemmaminimax1} still holds. Thus, we say that a collection $\nu$ of connected components of $Q_0\cap Q_1$ is an \textit{isolating family} for $\Omega$ 
if $\Omega\in \nu$ and there exists $\epsilon >0$ such that $B_\epsilon (\nu)\cap \Omega'=\emptyset$ for every connected component $\Omega'$ of $Q_0\cap Q_1$
that is not contained in $\nu$. Notice that the union of all connected components of $Q_0\cap Q_1$ is an isolating family for every connected component (in particular, 
isolating families always exist) and that $\nu =\{\Omega\}$ is an isolating family if $\Omega$ is isolated.

With slight abuse of notation we denote with $\nu$ both the isolating family and the union of all sets in the isolating family. 
For every isolating family $\nu$ for $\Omega$ we define 
$$k_\nu := \min \left \{c(L;Q_0,Q_1), \max_{q\in \nu} E(q,0)  + \max_{q\in \nu} \frac{|\theta_q|^2}{4a}\right \}.$$
It is easy to see now that the argument in the proof of Lemma \ref{lemmaminimax1} goes through replacing $k_\Omega$ with $k_\nu$. Since this holds for every isolating family
we can define 
\begin{equation}
k_\Omega:=\inf \left \{k_\nu \ \Big |\ \nu \text{ isolating family for } \Omega \right \}.
\label{komega}
\end{equation}

It is easy to see that this definition of $k_\Omega$ coincides with the one given in \eqref{kq0q1} if $\Omega$ is isolated. Indeed, there clearly holds $k_{\nu_1}\leq k_{\nu_2}$ 
if $\nu_1\subseteq \nu_2$ (meaning that every connected component of $Q_0\cap Q_1$ contained in $\nu_1$ is also contained in $\nu_2$), and $\{\Omega\}$ is the smallest 
isolating family for $\Omega$ if $\Omega$ is isolated.

The next lemma follows now directly from Lemma \ref{lemmaminimax1}, keeping in mind the new definition of $k_\Omega$.

\begin{lemma}
Let $Q_0,Q_1\subseteq M$ be two closed connected submanifolds with non-empty intersection, let $\Omega$ be a connected component of $Q_0\cap Q_1$ and let $k_{\Omega}$ be as  defined in \eqref{komega}. For every $k\in (k_{\Omega},c(L;Q_0,Q_1))$ there exists 
$\alpha>0$ such that 
$$\inf_{u\in \Gamma} \max_{s\in [0,1]} \A_k(u(s)) \geq \alpha.$$
\end{lemma}

\noindent We can now define the minimax function
\begin{equation}
c_\Omega:\big (k_\Omega, c(L;Q_0,Q_1 )\big ) \longrightarrow \R,\quad  c_\Omega(k) := \inf_{u\in \Gamma} \max_{[0,1]} \A_k\circ u.
\label{minimaxfunction1}
\end{equation}

Lemma \ref{lemmaminimax1}  above implies that $c_\Omega(k)>0$ for all $k$; furthermore, the monotonicity of $\A_k$ in $k$ implies that the minimax function $c_\Omega(\cdot)$ is monotonically increasing and hence almost everywhere differentiable. In Lemma \ref{struwe}  we
prove the existence of bounded Palais-Smale sequences for every value of the parameter $k$ at which the minimax functions $c_\Omega(\cdot)$ is differentiable, thus overcoming the lack of the Palais-Smale 
condition for $\A_k$ for subcritical energies. The proof is analogous to the one in the periodic case (see \cite{Con06} and \cite{Abb13} for further details) and is based 
on the celebrated \textit{Struwe monotonicity argument}  (cf. \cite{Str90}).

\begin{lemma}
Suppose that $\bar{k}$ is a point of differentiability for the minimax function $c_\Omega(\cdot)$  in (\ref{minimaxfunction1}). Then $\A_{\bar k}$ admits a bounded Palais-Smale sequence at level $c_\Omega(\bar{k})$.
\label{struwe} 
\end{lemma}

\begin{proof}
Since $\bar k$ is a point of differentiability for $c_\Omega(\cdot)$ we have 
\begin{equation}
|c_\Omega(k)-c_\Omega(\bar k)|\leq M |k-\bar k|
\label{modulocontinuita}
\end{equation}
for all $k$ sufficiently close to $\bar k$, where $M>0$ is a suitable constant. Let $\{k_h\}$ be a strictly decreasing sequence which converges to 
$\bar{k}$ and set $\epsilon_h:=k_h -\bar{k} \downarrow 0$. For every $h\in \N$ choose $u_h\in \Gamma$ (or $\Gamma_{Q_0\cap Q_1}$) such that 
$$\max_{u_h} \A_{k_h} \leq c_\Omega(k_h) + \epsilon_h.$$

Up to ignoring a finite numbers of $k_h$'s we may suppose that equation (\ref{modulocontinuita}) is satisfied by every $k_h$. If $z=(x,T)\in u_h$ is such that $\A_{\bar{k}}(z)>c_\Omega(\bar{k})-\epsilon_h$, then 
$$T = \frac{\A_{k_h}(z)-\A_{\bar{k}}(z)}{k_h-\bar{k}}\leq \frac{c_\Omega(k_h)+\epsilon_h - c_\Omega(\bar{k}) + \epsilon_h}{\epsilon_h} \leq M+2.$$
Moreover, 
$$\A_{\bar{k}}(z)\leq \A_{k_h}(z) \leq c_\Omega(k_h)+\epsilon_h \leq c_\Omega(\bar{k}) + (M+1)\epsilon_h$$
and hence
$$u_h  \subseteq  \mathcal A_h  \cup  \Big \{\A_{\bar{k}}  \leq  c_\Omega(\bar{k}) - \epsilon_h\Big \} ,$$
where 
$$\mathcal A_h= \Big \{ (x,T) \in \mathcal N \ \Big |\ T \leq  M+2 ,\ \A_{\bar{k}}(x,T)\leq c_\Omega(\bar{k}) + (M+1)\epsilon_h\Big \} .$$ 
Observe that, if $(x,T)\in \mathcal A_h$, then by (\ref{firstinequality}) we have
$$\A_{\bar{k}}(x,T) \geq \frac{a}{M+2}  \big \|x'\big \|_2^2 - (M+2)\big |b-\bar{k}\big |$$
and hence 
$$\big \|x'\big \|_2^2 \leq \frac{M+2}{a} \Big ( c_\Omega(\bar{k}) + (M+1)\epsilon_h + (M+2) \big |b - \bar{k}\big |\Big ),$$
which shows that $\mathcal A_h$ is bounded in $\mathcal N$, uniformly in $h$. Let $\Phi$ be the flow of the vector field obtained by multiplying $-\nabla \A_{\bar{k}}$ by a suitable non-negative function, whose role is to make the vector field bounded on $\mathcal N$ and vanishing on the sublevel $\big \{\A_{\bar{k}}\leq c_\Omega(\bar{k})/2 \big\}$, while keeping the uniform decrease condition 
\begin{equation}
\frac{d}{d\sigma} \A_{\bar{k}}(\Phi_\sigma (z)) \leq - \frac{1}{2}  \min \Big \{ \big \|d\A_{\bar{k}}(\Phi_\sigma(z))\big \|^2, 1\Big \} ,\quad\text{if}\ \A_{\bar{k}} (\Phi_\sigma (z)) \geq \frac{c_\Omega(\bar{k})}{2} .
\label{decrescita}
\end{equation}

Lemma \ref{convergenza} implies that $\Phi$ is well-defined on $[0,+\infty)\times \mathcal N$ and that $\Gamma$ is positively invariant with respect to $\Phi$. Since $\Phi$ maps bounded sets into bounded sets, 
\begin{equation}
\Phi\big ([0,1]\times u_h\big ) \subseteq \mathcal B_h \cup \Big \{\A_{\bar{k}} \leq c_\Omega(\bar{k})-\epsilon_h\Big \}
\label{fi}
\end{equation}
for some uniformly bounded set 
\begin{equation}
\mathcal B_h \subseteq \Big \{ \A_{\bar{k}} \leq c_\Omega(\bar{k}) + (M+1)\epsilon_h\Big \} .
\label{Bh}
\end{equation}
We claim that there exists a sequence $\{z_h\}\subseteq \mathcal N$ with 
$$z_h  \in  \mathcal B_h \cap \Big \{ \A_{\bar{k}} \geq c_\Omega(\bar{k}) - \epsilon_h \Big \}$$
and $\big \|d\A_{\bar{k}}(z_h)\big \|$ infinitesimal. Such a sequence is clearly a bounded Palais-Smale sequence at level $c_\Omega(\bar{k})$. Assume by contradiction that there exists $\delta \in (0,1)$  such that
$$\big \|d\A_{\bar{k}}\big \| \geq  \delta,\quad \text{on}\  \mathcal B_h  \cap \Big \{ \A_{\bar{k}} \geq c_\Omega(\bar{k}) - \epsilon_h \Big \}$$
for every $h$ large enough. Together with (\ref{decrescita}), (\ref{fi}) and (\ref{Bh}), this implies that, for $h$ large enough, for any $z\in u_h$ such that 
$$\Phi \big ([0,1]\times \{z\} \big ) \subseteq \Big \{\A_{\bar{k}}\geq c_\Omega(\bar{k})-\epsilon_h\Big \}$$
there holds 
$$\A_{\bar{k}}(\Phi_1(z)) \leq \A_{\bar{k}}(z) - \frac{1}{2} \delta^2  \leq c_\Omega(\bar{k}) + (M+1)\epsilon_h - \frac{1}{2} \delta^2.$$ 
It follows that 
$$\max_{\Phi_1 (u_h)}\A_{\bar{k}}\leq c_\Omega(\bar{k}) - \epsilon_h$$ 
for $h$ large enough. Since $\Phi_1(u_h)\in \Gamma$, this contradicts the definition of $c(\bar{k})$.
\end{proof}

\begin{proof}[Proof of Theorem \ref{teosubcriticalintro}]
Follows combining Lemma \ref{struwe} above with Lemma \ref{lemmalimitatezza} and with the fact that a monotonically increasing function is differentiable almost everywhere.
\end{proof}

\begin{oss}
The minimax functions $c_{\Omega}$ do not provide in general different critical points of $\A_k$. The only convenience to pick one different minimax class for 
each connected component $\Omega$ of the intersection $Q_0\cap Q_1$ is that, taking the infimum over all $k_{\Omega}$, one gets an a priori better 
critical value and, hence, a sharper result.
\end{oss}


\section{Counterexamples} 

Throughout this section $\Sigma$ will be a closed connected orientable surface and $\widetilde \Sigma$ will be its universal cover. 
Consider the hyperbolic plane 
$$\HH := \Big \{(x_1,x_2) \in \R^2 \ \Big | \ x_2>0\Big\}$$
endowed with the Riemannian metric 
\begin{equation}
g_{(x_1,x_2)} := \frac{1}{x_2^2} \left ( dx_1^2 + dx_2^2\right ).
\label{hyperbolicmetric}
\end{equation}
We refer to \cite{BKS91} for generalities and properties of  $(\HH,g)$. We define
\begin{equation}
L:T\HH\longrightarrow \R, \quad L(q,v) = \frac12 |v|_q^2 + \theta_q(v);
\label{hyperboliclagrangian}
\end{equation}
where $\theta_{(x_1,x_2)} =\frac{dx_1}{x_2}$ is the ``canonical primitive'' of the standard area form 
$$\sigma = \frac{1}{x_2^2} dx_1\wedge dx_2.$$ 
It is well-known that $c(L)=\frac12$ (c.f. \cite[Section 5.2]{CFP10}). In fact, the Hamiltonian associated with $L$ is 
$$H(q,p) = \frac12 |p-\theta_q|^2$$
and hence \eqref{hamiltoniancharacterization} implies that
$$c(L) = \inf_{u\in C^\infty(\HH)} \sup_{q\in \HH} \frac12 |d_qu-\theta_q|^2 \leq \frac12 ,$$
as $|\theta_q|\equiv 1$. Computing the $(L+k)$-action of the clockwise arc-length parametrization $\gamma_r$ of a (hyperbolic) circle with radius $r$ yields the opposite inequality. Indeed, using 
$$l(\gamma_r) = 2\pi \sinh r , \quad \operatorname{Area} (D_r) = 2\pi \big (\cosh r - 1\big ) ,$$
we readily compute for the action of $\gamma_r$
$$S_k(\gamma_r) = \int_0^{l(\gamma_r)} \Big [\frac12 |\dot \gamma_r(t)|^2 + k \Big ]\, dt  + \int_{\gamma_r} \theta = \pi  \Big (k-\frac12 \Big)\, e^r + f(r),$$
with $f(r)$ uniformly bounded function of $r$. It follows that,  for every $k<\frac12$ 
$$S_k(\gamma_r)\longrightarrow -\infty$$
as $r$ goes to infinity, thus showing that $c(L)\geq \frac12$. The restriction of the Euler-Lagrange flow to the energy level set $E^{-1}(\frac12)$ is the celebrated \textit{horocycle flow} of Hedlund (cf. \cite{BKS91} and  \cite{Hed32}). 
Its peculiarity relies on the fact that, once projected to a compact quotient of $\HH$, it becomes \textit{minimal}, meaning that every orbit is dense. For $k<\frac12$, the Euler-Lagrange flow on $E^{-1}(k)$ is periodic 
and the projections of the orbits to $\HH$ describe circles with hyperbolic (thus, euclidean) radius going to zero as $k\rightarrow 0$.


\vspace{3mm}

\noindent \textbf{Orbits connecting two points.} In this subsection we show an approximated counterexample to Contreras' result \cite{Con06} about the existence of Euler-Lagrange orbits connecting two points $q_0\neq q_1 \in M$. 
Strictly speaking, for every $\epsilon>0$, we embed the flow of the Lagrangian in (\ref{hyperboliclagrangian}) into any surface $\Sigma$ in a suitable fashion. If the points $q_0$ and $q_1$ are chosen properly, then they cannot be connected by orbits with energy less than $c_u-\epsilon$.

Thus, consider the Euler-Lagrange flow on $T\HH$ associated to the Lagrangian in (\ref{hyperboliclagrangian}) and fix $\epsilon>0$, $q_0\in \HH$. We know that, for every $k<\frac12$, the restriction of the Euler-Lagrange flow to $E^{-1}(k)$ is periodic and orbits describe hyperbolic (hence, euclidean) circles with the same hyperbolic radius. If we denote by $\rho(q_0,v)$ the euclidean radius of the (projection of the unique) Euler-Lagrange 
orbit through $(q_0,v)$, then we readily have 
$$\rho := \max_{k \leq  \frac12 -\epsilon} \max_{|v|=k} \rho(q_0,v) < \infty .$$

Let now $B_1\subseteq B_2 \subseteq B_3$ be open connected sets containing $q_0$ such that all Euler-Lagrange orbits with energy less than 
$\frac12 -\epsilon$ starting from $q_0$ are entirely contained in $B_1$.
We extend $\theta|_{B_1}$ to be equal to zero outside $B_2$ using a suitable cut-off function and embed $B_3$ in $\Sigma$. The embedding induces a Riemannian metric on a subset $\mathcal U$ of $\Sigma$ which can be extended to a metric on the whole $\Sigma$ and also a 1-form on $\Sigma$ obtained simply by setting the pull-back of $\theta$ to be zero outside $\mathcal U$.
\begin{figure}[h]
\begin{center}
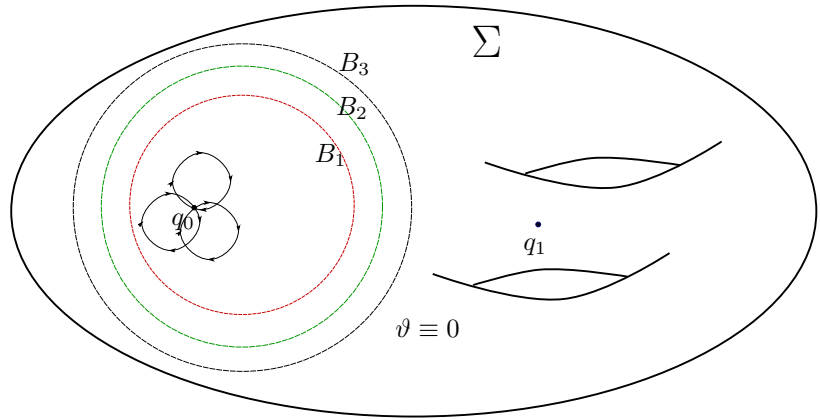
\end{center} 
\caption{An approximated counterexample to Contreras' result.}
\label{figure2}
\end{figure}

We denote the metric, the 1-form and the point on $\Sigma$ given by the embedding again with $g,\theta,q_0$ respectively and define the magnetic Lagrangian
$$L_\epsilon:T\Sigma \rightarrow \R , \quad L(q,v) = \frac12 \, |v|^2 + \theta_q(v) .$$

If we now consider $q_1\in \Sigma \setminus \mathcal U$, then by construction there are no Euler-Lagrange orbits with energy $k$ connecting $q_0$ with $q_1$ for every $k<\frac12 -\epsilon$. This automatically implies that $c_u(L_\epsilon)\geq \frac12 - \epsilon$ by Contreras' result.
At the same time, by \eqref{hamiltoniancharacterization},
$$c_u(L_\epsilon) = \inf_{u \in C^\infty(\widetilde \Sigma)} \sup_{\tilde q\in \widetilde \Sigma}  \frac12 |d_{\tilde q} u - \tilde \theta_{\tilde q}|^2 \leq \frac12 ,$$
where $\tilde \theta$ denotes the lift of $\theta$ to $\widetilde \Sigma$, since $|\theta_q|\leq 1$ for every $q\in \Sigma$.


\vspace{4mm}

\label{counterexamples}

\noindent \textbf{Supercritical energies.} In this subsection we prove Theorem \ref{counterexample1intro}. We also show that in general one cannot expect the existence of local minimizers (necessarily not global) for $\A_k$ in the energy range $(c_u(L),c(L;Q_0,Q_1))$, even if the configuration space is a surface; this is in sharp contrast with what happens in the case of periodic orbits (see e.g. \cite{AMMP14} or \cite{CMP04}).

The example constructed in the previous subsection suggests that below $c_u(L)$ we might not expect to find orbits connecting two given disjoint submanifolds. However in the example we gave above we had $c_u(L)=c(L;Q_0,Q_1)$. The natural question is now to study what happens for 
$$k\in \big (c_u(L),c(L;Q_0,Q_1)\big )\, .$$

In fact, for every energy in this range every point $q_0$ of $Q_0$ can be joined with every point $q_1$ of $Q_1$. Namely, for every $k$ in this energy range, the free-time action functional $\A_k$ on the space $\mathcal M_q$ of $H^1$-paths 
from $q_0$ to $q_1$ is bounded from below and satisfies the Palais-Smale condition; it follows that $\A_k$ has a global minimizer on each connected component of $\mathcal M_q$, which therefore corresponds to an Euler-Lagrange orbit from $q_0$ to $q_1$.
What is not clear is whether such an orbit connecting $q_0$ to $q_1$ also satisfies the conormal boundary conditions. 
Actually, it does not need to, as we now show. Namely, we exhibit an example of a magnetic Lagrangian and disjoint submanifolds $Q_0,Q_1$ such that $c_u(L)<c(L;Q_0,Q_1)$ and for every $k<c(L;Q_0,Q_1)$ there are no orbits satisfying the conormal boundary conditions. 
We shall start producing a situation where 
$$0< c_u(L) < c(L);$$
this is inspired by the construction in \cite{Man96}. Think of $\T^2$ as the square $[0,1]^2$ in $\R^2$ with identified sides and equipped with the euclidean metric and consider the magnetic Lagrangian 
\begin{equation}
L:T\T^2 \longrightarrow \R , \quad L (q,v) = \frac12  |v|^2 + \psi(y) v_x ,
\label{teorema1ottimale}
\end{equation}
where $q=(x,y)$, $v=(v_x,v_y)$ and $\psi:[0,1]\rightarrow [0,1]$ is a smooth cut-off function compactly supported in $(0,1)$ with $\psi \leq 1$, $\psi(\frac12)=1$, and $\psi'\geq 0$ 
on $[0,\frac 12]$, $\psi'\leq 0$ on $[\frac 12,1]$.

The Lagrangian in (\ref{teorema1ottimale}) is a magnetic Lagrangian with magnetic 1-form $\theta_q(\cdot) = \psi(y)dx$. Clearly $|\theta_q|=|\psi(y)|$ for every $q=(x,y)$ and hence 
$$c(L) = \inf_{u\in C^\infty(\T^2)} \max_{q\in \T^2} \ \frac12 \, |d_qu-\theta_q|^2 \leq \frac12 .$$

Conversely, consider the path $a:[0,1]\rightarrow \R^2,\, a(t) \ = \ (1-t\, ,\, \frac12)$; it is clear that $a$ is closed as a path in $\T^2$. We now readily compute for $k>0$ 
$$S_k(a) =  \int_0^1 \left (\frac12 \, |\dot a(t)|^2 + \psi(a(t))\dot a_x(t) +k \right )\, dt = \int_0^1 \left (\frac12  |1|^2  -1 +k \right )\, dt = k -  \frac12,$$
which is negative for every $k<\frac12$. We may then conclude that $c(L)=\frac12$.
\begin{figure}[h]
\begin{center}
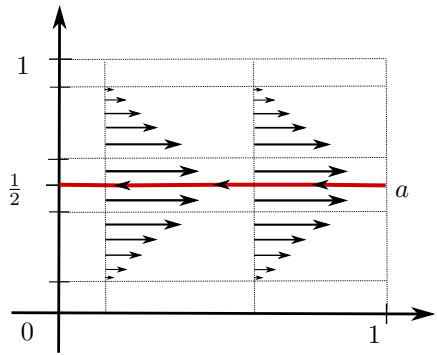
\end{center}
\label{figure3}
\caption{The Ma\~n\'e critical value $c(L)$ of the Lagrangian $L$ in \eqref{teorema1ottimale} equals $\frac 12$.}
\end{figure}

\noindent Again, by the Hamiltonian characterization of the Ma\~n\'e critical value we have 
$$c_u(L) = \inf_{u\in C^\infty(\R^2)} \sup_{q\in \R^2} \ \frac12 \, |d_qu-\theta_q|^2\leq \frac18$$
as one gets by choosing $u:\R^2\rightarrow \R$, $u(x,y) = \frac{x}{2}$:
$$\frac12  \left |\frac12 dx - \psi(y)dx\right |^2= \frac12  \left |\frac12 - \psi(y) \right|^2\leq \frac18 , \quad \forall (x,y)\in \R^2 ,$$
since $0 \leq \psi(y)\leq 1$ for every $y$. On the other hand, for every $n\in \N$ consider the contractible loop $\alpha_n := e\#d^n\#c\#b^n$
obtained concatenating the following paths $b,c,d,e$ with constant speed $\sqrt{2k}$:  
\begin{align*}
& b:[0,\frac{1}{\sqrt{2k}}]\rightarrow \R^2, \quad t\mapsto (1-\sqrt{2k}\, t, \frac 12)\\
& c:[0,\frac{1}{2\sqrt{2k}}]\rightarrow \R^2, \quad t \mapsto (0 , \frac 12 - \sqrt{2k} \, t )\\
& d:[0,\frac{1}{\sqrt{2k}}]\rightarrow \R^2, \quad t\mapsto (\sqrt{2k}\, t, 0)\\
& e:[0,\frac{1}{2\sqrt{2k}}]\rightarrow \R^2, \quad t\mapsto (1,\sqrt{2k}\, t).
\end{align*}
A straightforward computation shows that, for $k<\frac 18$,  
\begin{align*}
S_k(\alpha_n) & = n\cdot S_k(b)+S_k(c)+n \cdot S_k(d)+S_k(e) \\
                        & = \frac{n}{\sqrt{2k}} (2k - \sqrt{2k}) + \frac{2k}{2\sqrt{2k}} + \frac{n}{\sqrt{2k}} (2k)+  \frac{2k}{2\sqrt{2k}} \\
                        & = n \left (2\sqrt{2k} - 1 \right ) + \sqrt{2k}
                        \end{align*}
goes to $-\infty$ as $n\rightarrow +\infty$. It follows that $c_u(L)=\frac 18$.

We pick now $Q_0$ to be any point in $\T^2$, for instance $(\frac12,0)$ and $Q_1$ to be the circle $\{y=\frac12\}$; by construction we have 
$$c(L;Q_0,Q_1) = c(L)= \frac12 .$$
Since $\pi_0(\mathcal M_Q)\cong \Z$, Theorem \ref{teosupercriticalintro} implies that for every $k>\frac 12$ there are infinitely many Euler-Lagrange orbits with energy $k$ satisfying the conormal boundary conditions. Namely, there is one such orbit for every connected component of $\mathcal M_Q$, which is in addition a global minimizer of  $\A_k$ on its connected component. 
Furthermore, for every $k\in \big (\frac 18,\frac 12 \big )$ and any point $q_1\in Q_1$ there are infinitely many Euler-Lagrange orbits with energy $k$ joining $q_0=Q_0$ with $q_1$. However, none of these can satisfy the conormal boundary conditions for $Q_1$ since by the obstruction (\ref{necessarylag}) this is possible only above energy $\frac12$ 
(observe indeed that in this example $k(L;Q_0,Q_1)=c(L;Q_0,Q_1)=\frac 12$). 

The same counterexample holds clearly for every point of the form $q_0=(\frac12,h)$ for every $h>0$, in particular showing that we might not expect to find Euler-Lagrange orbits 
with energy less than $c(L;Q_0,Q_1)$ satisfying the conormal boundary conditions, even if the two submanifolds are ``close'' to each other.
Notice that this example for $q_0=(\frac12,\frac12)$ is not in contradiction with Theorem \ref{teosubcriticalintro}, since in this case we have 
$$ k_{Q_0\cap Q_1} = c(L;Q_0,Q_1) .$$

\vspace{2mm}

The Lagrangian in \eqref{teorema1ottimale} gives also a sharp counterexample to Contreras' result. Namely, if $q_0=(\frac 12,0)$ and $q_1=(\frac 12,\frac 12)$, then 
there are no Euler-Lagrange orbits connecting the two points for $k\leq \frac 18 =c_u(L)$. This can be seen as follows: The function $I:T\T^2\rightarrow \R$ given by 
$I(q,v) := v_x+\psi(y)$
is an integral of the motion. Computing $I(q_0,v)$ and $I(q_1,v)$ for every $v\in \R^2$ with $|v|=\sqrt{2k}$ yields 
$$I(q_0,v)=v_x\in [-\sqrt{2k},\sqrt{2k}],\quad I(q_1,v)=v_x+1 \in [1-\sqrt{2k},1+\sqrt{2k}].$$
Since the two intervals are disjoint for $k<\frac 18$, there are no orbits connecting the two points with energy $k<\frac 18$. Moreover, for $k=\frac 18$, the only possible
Euler-Lagrange orbit connecting the two points must start and end parallel to the $x$-axis. However, this is not possible because the orbits starting from $q_0$ or $q_1$ 
with tangent vector parallel to the $x$-axis are periodic orbits parallel to the $x$-axis. 
Note moreover that, for $k=\frac 18$, there are heteroclinic orbits connecting the two periodic 
orbits $\{y=0\}$ and $\{y=\frac 12\}$, provided that the function $\psi$ has non-degenerate minimum and maximum at $y=0$ and $y=\frac 12$, respectively.

\vspace{2mm}

We finally observe that the Lagrangian in (\ref{teorema1ottimale}) can also be used to give an example in which the interval $(c(L;Q_0,Q_1),k_{\mathcal N}(L))$ as in the statement 2-(a) of Theorem \ref{teosupercriticalintro} is non-empty. 
Just consider as $Q_0$ and $Q_1$ two small (contractible) intersecting circles with center on $\{y=\frac12\}$. 

\begin{figure}[h]
\begin{center}
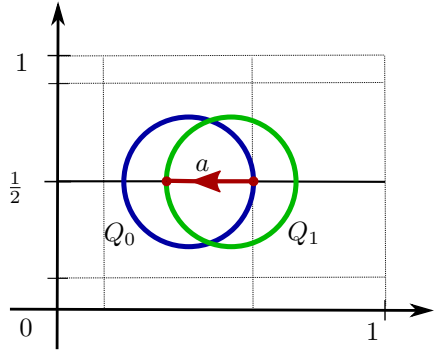
\end{center}
\caption{An example where $c(L;Q_0,Q_1)<k_{\mathcal N}(L)$.}
\label{figure4}
\end{figure}

\noindent One readily sees that for such a choice of submanifolds there holds 
$$c(L;Q_0,Q_1)  = c_u(L)  <  c(L)  = k_{\mathcal N}(L) =  \frac12 .$$ 
Indeed, the path $a:J\rightarrow \T^2$ with constant speed 1 depicted in Figure \ref{figure4} lies in the connected component $\mathcal N$ containing the constant paths and satisfies $\A_k(a)=|J| (k-\frac 12)$; this shows that $k_{\mathcal N}(L)\geq \frac 12$ and hence automatically $k_{\mathcal N}(L)=\frac 12$, as on the other hand by \eqref{chainofcriticalvalues} we have 
$$k_{\mathcal N}(L)\leq e_0(L)+\frac{\|\theta\|_\infty^2}{4a} = \frac{\|\theta\|_\infty^2}{2}=\frac 12,$$
since $e_0(L)=0$ and $a=\frac 12$ for magnetic Lagrangians, and $\|\theta\|_\infty=1$.

\vspace{5mm}

\noindent \textbf{Subcritical energies.} In this subsection we prove Theorem \ref{counterexample2intro} by constructing an example of magnetic Lagrangian $L:T\Sigma\rightarrow \R$ 
and intersecting submanifolds $Q_0,Q_1$ such that $c_u(L)\in [\frac 32,2]$, $k_{Q_0\cap Q_1}=\frac 12$, $k(L;Q_0,Q_1)=0$, and there are no connecting orbits with energy less than $\frac 12$.

We start considering the 1-form $2\theta$ on $\HH$ and the associated magnetic Lagrangian $L(q,v)=\frac 12|v|_q^2 + 2\theta_q(v)$; here $\theta$ is the canonical
primitive of the standard area form in $\HH$. It is easy to see that $c_u(L)=2$. 
Now let $Q_0$ be a rounded up rectangle in $\HH$; to fix the notation say that the vertical sides of $Q_0$ have $x=a$ and $x=b$ respectively. Moreover, fix $c<d$,
with $c,d\in [a,b]$, such that 
all orbits with energy $k\leq \frac 12$ and starting on the vertical sides stay in the region $\{x\notin [c,d]\}$. Up to increasing the length of the horizontal sides of $Q_0$ one 
sees that $c$ and $d$ actually exist. Take now a proper subinterval $[e,f]\subset [c,d]$ and let $\phi:\R\rightarrow [1,2]$ be a smooth function with the following properties:
\begin{itemize}
\item $\phi\equiv 2$ on $\R\setminus (c,d)$,
\item $\phi\equiv 1$ on $[e,f]$,
\item $\phi$ decreasing on $[c,e]$ and increasing on $[f,d]$.
\end{itemize}

Now consider $\theta'=\phi(x)\cdot \theta$, the associated magnetic Lagrangian $L'$ and set $Q_1$ to be any circle intersecting $Q_0$ and contained in 
the region $\{x\in [e,f]\}$ (see Figure \ref{figure5} below). Take $C\subset C'\subset \HH$ sufficiently large compact sets 
as in Figure \ref{figure5} and such that $C$ contains loops with negative $(L'+\frac 32)$-action (observe that this is possible since $c_u(L')=2$). 
Finally, using a suitable cut-off function, define from $\theta'$ a new 1-form $\theta''$ on $\HH$ such that $\theta''\equiv\theta'$ on $C$ and $\theta''\equiv0$ outside $C'$,
and consider the associated magnetic Lagrangian $L''$. 

Embedding this example into any surface $\Sigma$ as done at the beginning of this section yields 
a magnetic Lagrangian $L:T\Sigma\rightarrow \R$ and intersecting submanifolds $Q_0$,$Q_1$ such that: 
\begin{itemize}
\item $c_u(L)\in (\frac 32,2]$ and $k_{Q_0\cap Q_1}=\frac 12$. 
\item For almost every $k\in (\frac 12, c_u(L))$ there is a connecting orbit with energy $k$ by Theorem \ref{teosubcriticalintro}.
\item For every $k<\frac 12$ there are no connecting orbits with energy $k$. Indeed, since $|\mathcal P_0 w_q|\geq 1$ on the horizontal edges of 
$Q_0$, by the obstruction \eqref{necessarylag} such connecting orbits should start from the vertical edges of $Q_0$. However, by construction the 
orbits starting from the vertical edges of $Q_0$ do not intersect the region $\{x\in [e,f]\}$ and, hence, they cannot be connecting orbits. 
\item $k(L;Q_0,Q_1)=0$. In particular, the condition $k>k(L;Q_0,Q_1)$ 
is not sufficient to guarantee the existence of connecting orbits. 
\end{itemize}

\begin{figure}[h]
\begin{center}
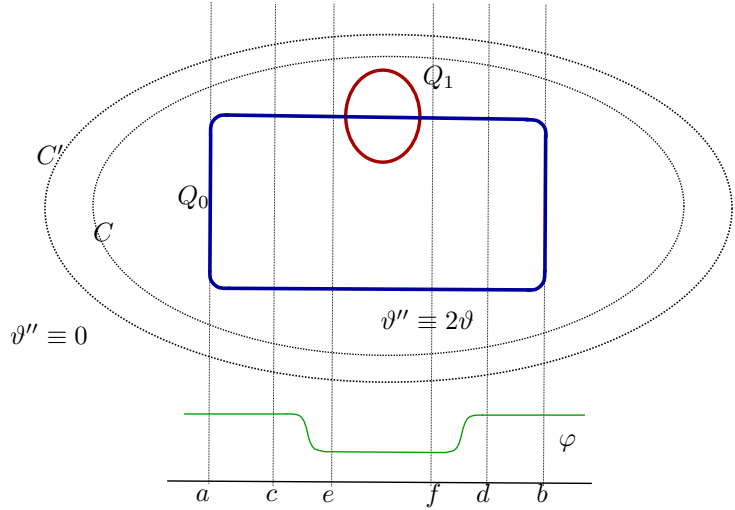
\end{center}
\caption{Theorem \ref{teosubcriticalintro} is sharp.}
\label{figure5}
\end{figure}

%

\vspace{3mm}

\textit{Acknowledgments.} I warmly thank Alberto Abbondandolo for drawing my attention on the subject. I am also grateful to the anonimous referee for his very careful reading of the draft and 
for his precious suggestions which helped me to improve essentially the paper. The author is partially supported by the DFG grant AB 360/2-1 ``Periodic orbits
of conservative systems below the Ma\~n\'e critical energy value''. 

\bibliographystyle{amsplain}


\end{document}

%% file: controesempio.pdf_tex
\begingroup%
  \makeatletter%
  \providecommand\color[2][]{%
    \errmessage{(Inkscape) Color is used for the text in Inkscape, but the package 'color.sty' is not loaded}%
    \renewcommand\color[2][]{}%
  }%
  \providecommand\transparent[1]{%
    \errmessage{(Inkscape) Transparency is used (non-zero) for the text in Inkscape, but the package 'transparent.sty' is not loaded}%
    \renewcommand\transparent[1]{}%
  }%
  \providecommand\rotatebox[2]{#2}%
  \ifx\svgwidth\undefined%
    \setlength{\unitlength}{263.08774465bp}%
    \ifx\svgscale\undefined%
      \relax%
    \else%
      \setlength{\unitlength}{\unitlength * \real{\svgscale}}%
    \fi%
  \else%
    \setlength{\unitlength}{\svgwidth}%
  \fi%
  \global\let\svgwidth\undefined%
  \global\let\svgscale\undefined%
  \makeatother%
  \begin{picture}(1,0.40064488)%
    \put(0,0){\includegraphics[width=\unitlength,page=1]{controesempio.pdf}}%
    \put(0.09982993,0.1760775){\color[rgb]{0,0,0}\makebox(0,0)[lb]{\smash{$Q_0$}}}%
    \put(0.61033426,0.17353738){\color[rgb]{0,0,0}\makebox(0,0)[lb]{\smash{$Q_1$}}}%
    \put(0.61680571,0.04656625){\color[rgb]{0,0,0}\makebox(0,0)[lb]{\smash{$q_0$}}}%
    \put(0.81787314,0.24801206){\color[rgb]{0,0,0}\makebox(0,0)[lb]{\smash{$\mathbb T^2$}}}%
  \end{picture}%
\endgroup%

%% file: cc.pdf_tex
\begingroup%
  \makeatletter%
  \providecommand\color[2][]{%
    \errmessage{(Inkscape) Color is used for the text in Inkscape, but the package 'color.sty' is not loaded}%
    \renewcommand\color[2][]{}%
  }%
  \providecommand\transparent[1]{%
    \errmessage{(Inkscape) Transparency is used (non-zero) for the text in Inkscape, but the package 'transparent.sty' is not loaded}%
    \renewcommand\transparent[1]{}%
  }%
  \providecommand\rotatebox[2]{#2}%
  \ifx\svgwidth\undefined%
    \setlength{\unitlength}{362.30114325bp}%
    \ifx\svgscale\undefined%
      \relax%
    \else%
      \setlength{\unitlength}{\unitlength * \real{\svgscale}}%
    \fi%
  \else%
    \setlength{\unitlength}{\svgwidth}%
  \fi%
  \global\let\svgwidth\undefined%
  \global\let\svgscale\undefined%
  \makeatother%
  \begin{picture}(1,0.33645064)%
    \put(0,0){\includegraphics[width=\unitlength,page=1]{cc.pdf}}%
    \put(0.46697968,0.26540653){\color[rgb]{0,0,0}\makebox(0,0)[lb]{\smash{$Q_0$}}}%
    \put(0.34183041,0.07197283){\color[rgb]{0,0,0}\makebox(0,0)[lb]{\smash{$Q_1$}}}%
    \put(0,0){\includegraphics[width=\unitlength,page=2]{cc.pdf}}%
    \put(0.41106296,0.24113271){\color[rgb]{0,0,0}\makebox(0,0)[lb]{\smash{$q_0$}}}%
    \put(0.41500599,0.07632692){\color[rgb]{0,0,0}\makebox(0,0)[lb]{\smash{$q_1$}}}%
    \put(0.56132622,0.23955547){\color[rgb]{0,0,0}\makebox(0,0)[lb]{\smash{$q_2$}}}%
    \put(0.56605787,0.07632692){\color[rgb]{0,0,0}\makebox(0,0)[lb]{\smash{$q_3$}}}%
    \put(0.35975323,0.2981121){\color[rgb]{0,0,0}\makebox(0,0)[lb]{\smash{$\mathbb T^2$}}}%
    \put(0,0){\includegraphics[width=\unitlength,page=3]{cc.pdf}}%
    \put(0.48227902,0.19348386){\color[rgb]{0,0,0}\makebox(0,0)[lb]{\smash{$a$}}}%
  \end{picture}%
\endgroup%

%% file: controesempioiperbolico.pdf_tex
\begingroup%
  \makeatletter%
  \providecommand\color[2][]{%
    \errmessage{(Inkscape) Color is used for the text in Inkscape, but the package 'color.sty' is not loaded}%
    \renewcommand\color[2][]{}%
  }%
  \providecommand\transparent[1]{%
    \errmessage{(Inkscape) Transparency is used (non-zero) for the text in Inkscape, but the package 'transparent.sty' is not loaded}%
    \renewcommand\transparent[1]{}%
  }%
  \providecommand\rotatebox[2]{#2}%
  \ifx\svgwidth\undefined%
    \setlength{\unitlength}{306.05040016bp}%
    \ifx\svgscale\undefined%
      \relax%
    \else%
      \setlength{\unitlength}{\unitlength * \real{\svgscale}}%
    \fi%
  \else%
    \setlength{\unitlength}{\svgwidth}%
  \fi%
  \global\let\svgwidth\undefined%
  \global\let\svgscale\undefined%
  \makeatother%
  \begin{picture}(1,0.50851763)%
    \put(0,0){\includegraphics[width=\unitlength,page=1]{controesempioiperbolico.pdf}}%
    \put(-0.33224485,-0.59449118){\color[rgb]{0,0,0}\makebox(0,0)[lb]{\smash{}}}%
    \put(0.56801515,0.44411961){\color[rgb]{0,0,0}\makebox(0,0)[lb]{\smash{$\text{\huge{$\Sigma$}}$ }}}%
    \put(0.19853656,0.23765533){\color[rgb]{0,0,0}\makebox(0,0)[lb]{\smash{$q_0$}}}%
    \put(0.37499319,0.31541368){\color[rgb]{0,0,0}\makebox(0,0)[lb]{\smash{$B_1$}}}%
    \put(0.40125816,0.37236097){\color[rgb]{0,0,0}\makebox(0,0)[lb]{\smash{$B_2$}}}%
    \put(0.40403291,0.42651955){\color[rgb]{0,0,0}\makebox(0,0)[lb]{\smash{$B_3$}}}%
    \put(0.47446381,0.09901377){\color[rgb]{0,0,0}\makebox(0,0)[lb]{\smash{$\theta\equiv 0$}}}%
    \put(0,0){\includegraphics[width=\unitlength,page=2]{controesempioiperbolico.pdf}}%
    \put(0.63177337,0.20693078){\color[rgb]{0,0,0}\makebox(0,0)[lb]{\smash{$q_1$}}}%
  \end{picture}%
\endgroup%

%% file: teorema1ottimale.pdf_tex
\begingroup%
  \makeatletter%
  \providecommand\color[2][]{%
    \errmessage{(Inkscape) Color is used for the text in Inkscape, but the package 'color.sty' is not loaded}%
    \renewcommand\color[2][]{}%
  }%
  \providecommand\transparent[1]{%
    \errmessage{(Inkscape) Transparency is used (non-zero) for the text in Inkscape, but the package 'transparent.sty' is not loaded}%
    \renewcommand\transparent[1]{}%
  }%
  \providecommand\rotatebox[2]{#2}%
  \ifx\svgwidth\undefined%
    \setlength{\unitlength}{172.30690322bp}%
    \ifx\svgscale\undefined%
      \relax%
    \else%
      \setlength{\unitlength}{\unitlength * \real{\svgscale}}%
    \fi%
  \else%
    \setlength{\unitlength}{\svgwidth}%
  \fi%
  \global\let\svgwidth\undefined%
  \global\let\svgscale\undefined%
  \makeatother%
  \begin{picture}(1,0.76123054)%
    \put(0,0){\includegraphics[width=\unitlength,page=1]{teorema1ottimale.pdf}}%
    \put(-0.00721033,0.34301035){\color[rgb]{0,0,0}\makebox(0,0)[lb]{\smash{$\frac{1}{2}$}}}%
    \put(0.02450298,0.02781652){\color[rgb]{0,0,0}\makebox(0,0)[lb]{\smash{$0$}}}%
    \put(0.78296026,0.0203534){\color[rgb]{0,0,0}\makebox(0,0)[lb]{\smash{$1$}}}%
    \put(0.01526673,0.60988927){\color[rgb]{0,0,0}\makebox(0,0)[lb]{\smash{$1$}}}%
    \put(0.84319649,0.34099567){\color[rgb]{0,0,0}\makebox(0,0)[lb]{\smash{$a$}}}%
  \end{picture}%
\endgroup%

%% file: teorema1ottimale2.pdf_tex
\begingroup%
  \makeatletter%
  \providecommand\color[2][]{%
    \errmessage{(Inkscape) Color is used for the text in Inkscape, but the package 'color.sty' is not loaded}%
    \renewcommand\color[2][]{}%
  }%
  \providecommand\transparent[1]{%
    \errmessage{(Inkscape) Transparency is used (non-zero) for the text in Inkscape, but the package 'transparent.sty' is not loaded}%
    \renewcommand\transparent[1]{}%
  }%
  \providecommand\rotatebox[2]{#2}%
  \ifx\svgwidth\undefined%
    \setlength{\unitlength}{161.5938027bp}%
    \ifx\svgscale\undefined%
      \relax%
    \else%
      \setlength{\unitlength}{\unitlength * \real{\svgscale}}%
    \fi%
  \else%
    \setlength{\unitlength}{\svgwidth}%
  \fi%
  \global\let\svgwidth\undefined%
  \global\let\svgscale\undefined%
  \makeatother%
  \begin{picture}(1,0.81169744)%
    \put(0,0){\includegraphics[width=\unitlength,page=1]{teorema1ottimale2.pdf}}%
    \put(0.0220707,0.02966066){\color[rgb]{0,0,0}\makebox(0,0)[lb]{\smash{$0$}}}%
    \put(0.83081103,0.02170276){\color[rgb]{0,0,0}\makebox(0,0)[lb]{\smash{$1$}}}%
    \put(0.01222213,0.65032278){\color[rgb]{0,0,0}\makebox(0,0)[lb]{\smash{$1$}}}%
    \put(0,0){\includegraphics[width=\unitlength,page=2]{teorema1ottimale2.pdf}}%
    \put(0.21880428,0.2527464){\color[rgb]{0,0,0}\makebox(0,0)[lb]{\smash{$Q_0$}}}%
    \put(0.6474824,0.25404641){\color[rgb]{0,0,0}\makebox(0,0)[lb]{\smash{$Q_1$}}}%
    \put(0.43299154,0.41523959){\color[rgb]{0,0,0}\makebox(0,0)[lb]{\smash{$a$}}}%
    \put(-0.00719155,0.3616383){\color[rgb]{0,0,0}\makebox(0,0)[lb]{\smash{$\frac{1}{2}$}}}%
  \end{picture}%
\endgroup%

%% file: controesempiointersezione2.pdf_tex
\begingroup%
  \makeatletter%
  \providecommand\color[2][]{%
    \errmessage{(Inkscape) Color is used for the text in Inkscape, but the package 'color.sty' is not loaded}%
    \renewcommand\color[2][]{}%
  }%
  \providecommand\transparent[1]{%
    \errmessage{(Inkscape) Transparency is used (non-zero) for the text in Inkscape, but the package 'transparent.sty' is not loaded}%
    \renewcommand\transparent[1]{}%
  }%
  \providecommand\rotatebox[2]{#2}%
  \ifx\svgwidth\undefined%
    \setlength{\unitlength}{282.96764362bp}%
    \ifx\svgscale\undefined%
      \relax%
    \else%
      \setlength{\unitlength}{\unitlength * \real{\svgscale}}%
    \fi%
  \else%
    \setlength{\unitlength}{\svgwidth}%
  \fi%
  \global\let\svgwidth\undefined%
  \global\let\svgscale\undefined%
  \makeatother%
  \begin{picture}(1,0.67016371)%
    \put(0.75186365,0.5480332){\color[rgb]{0.66666667,0,0}\makebox(0,0)[lb]{\smash{}}}%
    \put(0.83464679,0.51774069){\color[rgb]{0.66666667,0,0}\makebox(0,0)[lb]{\smash{}}}%
    \put(0.14358347,0.31685122){\color[rgb]{0,0,0}\makebox(0,0)[lb]{\smash{}}}%
    \put(0,0){\includegraphics[width=\unitlength,page=1]{controesempiointersezione2.pdf}}%
    \put(0.24346711,0.00448196){\color[rgb]{0,0,0}\makebox(0,0)[lb]{\smash{$a$}}}%
    \put(0.33681181,0.00448196){\color[rgb]{0,0,0}\makebox(0,0)[lb]{\smash{$c$}}}%
    \put(0.41170458,0.00339658){\color[rgb]{0,0,0}\makebox(0,0)[lb]{\smash{$e$}}}%
    \put(0.54955071,0.00339658){\color[rgb]{0,0,0}\makebox(0,0)[lb]{\smash{$f$}}}%
    \put(0.61684579,0.00231119){\color[rgb]{0,0,0}\makebox(0,0)[lb]{\smash{$d$}}}%
    \put(0.69716564,0.00231119){\color[rgb]{0,0,0}\makebox(0,0)[lb]{\smash{$b$}}}%
    \put(0,0){\includegraphics[width=\unitlength,page=2]{controesempiointersezione2.pdf}}%
    \put(0.48994614,0.23283784){\color[rgb]{0,0,0}\makebox(0,0)[lb]{\smash{$\theta'' \equiv 2\theta$}}}%
    \put(-0.00350581,0.21359647){\color[rgb]{0,0,0}\makebox(0,0)[lb]{\smash{$\theta''\equiv0$}}}%
    \put(0,0){\includegraphics[width=\unitlength,page=3]{controesempiointersezione2.pdf}}%
    \put(0.21902583,0.39959956){\color[rgb]{0,0,0}\makebox(0,0)[lb]{\smash{$Q_0$}}}%
    \put(0.54574507,0.56038255){\color[rgb]{0,0,0}\makebox(0,0)[lb]{\smash{$Q_1$}}}%
    \put(0.10668385,0.35012793){\color[rgb]{0,0,0}\makebox(0,0)[lb]{\smash{$C$}}}%
    \put(0.03144567,0.45319391){\color[rgb]{0,0,0}\makebox(0,0)[lb]{\smash{$C'$}}}%
    \put(0.72714121,0.07803367){\color[rgb]{0,0,0}\makebox(0,0)[lb]{\smash{$\phi$}}}%
  \end{picture}%
\endgroup%

%% file: On_the_existence_of_Euler-Lagrange_orbits....bbl
\begin{thebibliography}{9}
\bibitem{Abb13} A. Abbondandolo, \emph{Lectures on the free period {L}agrangian action functional}, J. Fixed Point Theory Appl., \textbf{13} (2013), 2, 397-430.
\bibitem{AMMP14} A. Abbondandolo  and L. Macarini and M. Mazzucchelli and G. P. Paternain, \emph{Infinitely many periodic orbits of exact magnetic flows on surfaces for almost every subcritical energy level} (2016), to appear in Journal of the European Mathematical Society.
\bibitem{AS09} A. Abbondandolo and M. Schwarz, \emph{A smooth pseudo-gradient for the {L}agrangian action functional}, Adv. Nonlinear Stud., \textbf{9} (2009), 4, 597-623.
\bibitem{Arn86} V. I. Arnold, \emph{The first steps of symplectic topology}, Russian Math. Surveys, \textbf{41} (1986), 1-21.
\bibitem{Ass15} L. Asselle, \emph{On the existence of orbits satisfying periodic or conormal boundary conditions for {E}uler-{L}agrange flows}, Ph.D. thesis, Ruhr Universit\"at Bochum (2015).
\bibitem{AB14}  L. Asselle  and G. Benedetti, G, \emph{The Lyusternik-Fet theorem for autonomous Tonelli Hamiltonian systems on twisted cotangent bundles}, 
J. Topol. Anal. 8 (2016), 3, 545-570.
\bibitem{AB15a} L. Asselle and G. Benedetti, \emph{Infinitely many periodic orbits of non-exact oscillating magnetic flows on surfaces with genus at least two for almost every low energy level}, Calc. Var. Partial Differential Equations, \textbf{52} (2015), 3-4.
\bibitem{Ban80} V. Bangert, \emph{Closed geodesics on complete surfaces}, Math. Ann. 251 (1980), 1, 83-96.
\bibitem{BKS91} T. Bedford and M. Keane and C. Series, \emph{Ergodic theory, symbolic dynamics and topological entropy}, Oxford Sci. Pub. (1991).
\bibitem{BP02} K. Burns and G. P. Paternain, \emph{Anosov magnetic flows, critical values and topological entropy}, Nonlinearity, \textbf{15} (2002), 2, 281-314.
\bibitem{Con06} G. Contreras, \emph{The Palais-Smale condition on contact type energy levels for convex {L}agrangian systems}, Calc. Var. Partial Differential Equations, \textbf{27} (2006), 3, 321-395.
\bibitem{CI99} G. Contreras and R. Iturriaga, \emph{Global minimizers of autonomous {L}agrangians}, 22° Col\'oquio Brasileiro de Matem\'atica [22nd Brazilian Mathematics Colloquium], Instituto de Matem\'atica Pura e Aplicada (IMPA), Rio de Janeiro (1999).
\bibitem{CFP10} K. Cieliebak and U. Frauenfelder and G. P. Paternain, \emph{Symplectic topology of {M}a\~n\'e's critical values}, Geom. Topol. 13 (2010), 3, 1765-1870.
\bibitem{CMP04} G. Contreras and L. Macarini and G. P. Paternain, \emph{Periodic orbits for exact magnetic flows on surfaces}, Int. Math. Res. Not. (2004), 8, 361-387.
\bibitem{Dui76} I. J. Duistermaat, \emph{On the {M}orse index in variational calculus}, Advances in Math., \textbf{21} (1976), 173-195.
\bibitem{FM07} A. Fathi and E. Maderna, \emph{Weak {K}{A}{M} theorem on non compact manifolds}, NoDEA Nonlinear Differential Equations Appl., \textbf{14} (2007), 1-27.
\bibitem{FL51} A. I. Fet and L.A. Lusternik, \emph{Variational problems on closed manifolds}, Doklady Akad. Nauk SSSR, \textbf{81} (1951), 17-18.
\bibitem{Gli97} Y. Gliklikh, \emph{Global analysis in mathematical physics}, Springer (1997).
\bibitem{Hat02} A. Hatcher, \emph{Algebraic Topology}, Cambridge University Press (2002).
\bibitem{Hed32} G. A. Hedlund, \emph{Geodesics on a two-dimensional Riemannian manifold with periodic coefficients}, Ann. of Math. (2), \textbf{33} (1932), 4, 719-739.
\bibitem{Hör90} L. H\"ormander, \emph{The analysis of partial differential operators i}, Springer (1990).
\bibitem{Kli78} W. Klingenberg, \emph{Lecture notes on closed geodesics}, Springer (1978).
\bibitem{Man96} R. Ma\~n\'e, \emph{Generic properties and problems of minimizing measures of {L}agrangian systems}, Nonlinearity, \textbf{6} (1996), 2, 273-310.
\bibitem{Man97} R. Ma\~n\'e, \emph{Lagrangian flows: the dynamics of globally minimizing orbits}, Bol. Soc. Brasil. Mat., \textbf{28} (1997), 2, 141-153.
\bibitem{Moh01} K. Mohnke, \emph{Holomorphic disks and the chord conjecture}, Ann. of Math. (2), \textbf{154} (2001), 219-222.
\bibitem{PP97} G. P. Paternain and M. Paternain, \emph{Critical values of autonomous {L}agrangian systems}, Comment. Math. Helv., \textbf{72} (1997), 3, 481-499.
\bibitem{Str90} M. Struwe, \emph{Existence of periodic solutions of {H}amiltonian systems on almost every energy surface}, Bull. Brasil. Mat. Soc., \textbf{20} (1990), 2, 49-58.
\end{thebibliography}
